\documentclass{article}
\usepackage{cite}
\usepackage[utf8]{inputenc}
\usepackage{float}
\usepackage{mathtools}
\usepackage{bm,booktabs,rotating}
\usepackage{amsmath}
\usepackage{amssymb}
\usepackage{graphicx}
\usepackage{wasysym}
\usepackage[unicode=true,pdfusetitle,
 bookmarks=true,bookmarksnumbered=false,bookmarksopen=false,
 breaklinks=false,pdfborder={0 0 1},backref=false,colorlinks=false]
 {hyperref}
\usepackage{multicol,soul}
\usepackage{multirow}
\usepackage{color}
\usepackage{setspace}

\usepackage{booktabs}
\usepackage{makecell}

\usepackage{mathtools, nccmath}

\usepackage{subcaption}
\floatstyle{ruled}
\newfloat{algorithm}{tbp}{loa}
\providecommand{\algorithmname}{Algorithm}
\floatname{algorithm}{\protect\algorithmname}

\usepackage{algorithm}
\usepackage{algorithmicx}
\usepackage{algpseudocode}
\newcommand{\algrule}[1][.2pt]{\par\vskip.5\baselineskip\hrule height #1\par\vskip.5\baselineskip}

\usepackage{tikz}
\usetikzlibrary{decorations.pathreplacing,calc}
\newcommand{\tikzmark}[1]{\tikz[overlay,remember picture] \node (#1) {};}

\makeatletter
    \algrenewcommand\alglinenumber[1]{\tikzmark{\arabic{ALG@line}}\tiny#1:}
\makeatother

\newtheorem{lemma}{Lemma}
\newtheorem{prop}{Proposition}

\newtheorem{theorem}{Theorem}\newenvironment{proof}{\sl \noindent Proof: \rm}{$\Box$}
\newcommand{\ignore}[1]{}

\newtheorem{remark}{Remark}

\usepackage{bm}
\newcommand{\M}[2][]{\bm{#1{\mathbf{\MakeUppercase{#2}}}}}              

\title{Minimizing the Arithmetic and Communication Complexity of Jacobi's Method for Eigenvalues and Singular Values: \\ Part One - Serial Algorithms}
\author{James Demmel, others?}
\author{James Demmel\footnote{Department of EECS (Computer Science Division) and Department of Mathematics, University of California Berkeley} \and Hengrui Luo\footnote{Department of Statistics, Rice University}
\and Ryan Schneider\footnote{Department of Mathematics, University of California Berkeley}
\and Yifu Wang\footnote{Committee on Computational and Applied Mathematics, University of Chicago}}
\date{}
\makeatother

\begin{document}
\maketitle
\begin{abstract}

We analyze several versions of Jacobi's method for the symmetric eigenvalue problem. Our goal is to reduce the asymptotic cost of the algorithm as much as possible, as measured by the number of arithmetic operations performed and associated (serial or parallel) communication, i.e., the amount of data moved between slow and fast memory or between processors in a network. The first half of this effort, which considers the serial setting, is presented here; this paper contains rigorous complexity bounds for a variety of serial Jacobi algorithms, built on both classic $O(n^3)$ matrix multiplication and fast, Strassen-like $O(n^{\omega_0})$ alternatives. In the classical case, we show that a blocked implementation of Jacobi's method attains the communication lower bound for $O(n^3)$ matrix multiplication (and is therefore expected to be communication optimal among $O(n^3)$ eigensolvers). In the fast setting, we demonstrate that a recursive version of blocked Jacobi can go further, reaching essentially optimal complexity in both measures. We also derive analogous complexity bounds for (one-sided) Jacobi SVD algorithms. A forthcoming sequel to this paper will extend our complexity analysis to the parallel case.
\end{abstract}

\section{Introduction}\label{sec:intro}
We study Jacobi's method \cite{Jacobi1846}, the oldest known algorithm for solving the symmetric eigenvalue problem. Starting from a (dense) symmetric matrix $\mathbf{A} \in {\mathbb R}^{n \times n}$, Jacobi's method applies a sequence of orthogonal similarity transformations, each of which annihilates a pair, or collection, of off-diagonal entries of $\mathbf{A}$ (see Section~\ref{section: classical} for a precise description). In doing so, the algorithm gradually diagonalizes the input matrix, producing a decomposition $\mathbf{A} = \mathbf{Q}\mathbf{D}\mathbf{Q}^T$, where $\mathbf{Q}$ is an orthogonal eigenvector matrix and $\mathbf{D}$ is a diagonal matrix containing the eigenvalues of $\mathbf{A}$. Jacobi's method, like other algorithms for the symmetric eigenproblem, can also be used to compute the singular value decomposition (SVD) of an arbitrary real matrix \cite{DrmacVeselic08a,DrmacVeselic08b}, recalling that the right singular vectors of $\mathbf{G} \in {\mathbb R}^{n \times n}$ diagonalize the symmetric matrix $\mathbf{G}^T\mathbf{G}$. Throughout, we restrict to real matrices for convenience; the analogous complex Hermitian case is a straightforward extension.   \\
\indent Alternatives to Jacobi's method include approaches that reduce $\mathbf{A}$ to tridiagonal form and subsequently diagonalize a tridiagonal 
matrix \cite{DHS89,Cuppen_dnc,sobczyk2024, dhillon2003orthogonal}, 
power iterations \cite{golub2000eigenvalue}, and highly-parallel spectral
divide-and-conquer algorithms \cite{Nakatsukasa_Higham_2013,Shah_Hermitian,gemp2020eigengame,definite_dnc}. In spite of these more modern options, Jacobi's method has endured, in part because of its simplicity and in part because of its well-documented performance advantages \cite{athi2016real,shiri2019fpga}, particularly the high relative accuracy it can achieve when applied to matrices with small eigenvalues (or small singular values in the SVD case) \cite{DemmelVeselic92,DrmacVeselic08a,DrmacVeselic08b}.  In light of its continued relevance, we derive in this paper rigorous complexity bounds for several versions of Jacobi's method, which estimate associated arithmetic and serial communication costs. In this context, arithmetic cost concerns the number of floating point operations (flops) performed, while communication refers to the cost of moving data to/from fast memory of size $M$. \\ 
\indent Importantly, diagonalizing a symmetric matrix $\mathbf{A}$ is at least as costly, in both arithmetic and communication, as matrix multiplication.\footnote{\textcolor{black}{See Appendix~\ref{section: lower_bounds}.}} In the serial case, our baselines for arithmetic/communication complexity 
are therefore $\Omega(n^{\omega_0})$ and $\Omega(n^{\omega_0}/M^{\omega_0/2-1})$, respectively, for $2 < \omega_0 \leq 3$.\footnote{\textcolor{black}{We assume familiarity with standard big-O notation; see the end of this section for definitions.}}  Here, the specific value of $\omega_0$ is determined by the matrix multiplication algorithm used by Jacobi. $\omega_0 = 3$ corresponds to classical dot-product-based matrix multiplication, for which the communication lower bound was proved by Hong and Kung \cite{HK81}, while $2 < \omega_0 < 3$ designates the wide class of \textit{fast} matrix multiplication algorithms \cite{blaser2013fast}. For these, we focus specifically on Strassen-like options -- i.e., those with recursive structure like that of Strassen's $O(n^{\log_2(7)})$ routine \cite{Strassen69}, which exhibit stability \cite{fastMM} and can be formulated to attain the associated communication lower bound \cite{BDHS14,SHS15}. 
This includes the current fastest-known algorithms \cite{faster_MM,fawzi2022discovering}. \\
\indent Our goal throughout is to probe the limits of Jacobi's method to  reach these complexity bounds when built on both classical \textit{and} fast matrix multiplication (hence, some results are stated in terms of $n^3$ while others are in terms of $n^{\omega_0}$). In doing so, we place Jacobi's method at the intersection of two concurrent efforts in numerical linear algebra, which aim to (stably) implement a variety of basic linear algebra computations with arithmetic complexity equal to that of matrix multiplication \cite{demmel2007fast} while also deriving algorithms that minimize associated communication costs \cite{ActaNumerica2014}. In the case of symmetric eigensolvers, these considerations go back at least to the late 80s (e.g., \cite{Yau_Lu}).

\renewcommand{\arraystretch}{1.6}
\renewcommand\multirowsetup{\centering}
\begin{table}[t]
    \centering
    \begin{tabular}{cccccc}
        \toprule
         \textbf{Type} & \textbf{Algorithm} & \textbf{Source} & \textbf{Arithmetic} & \textbf{Communication} & \textbf{Notes}  \\
         \midrule
         \multirow{2}{*}{\makecell{Tridiagonal\\ Reduction}} & Classical & \cite{DHS89,Cuppen_dnc,BDK12} & $O(n^3)$ & $O(n^3/\sqrt{M})$ &\\
         & Fast & \cite{sobczyk2024} & $O(n^{\omega_0}\log(n))$ & Open &  \\
         \midrule
         \multirow{2}{*}{\makecell{Divide-and-\\ Conquer}} & QDWH-eig & \cite{Nakatsukasa_Higham_2013} & $O(n^3)$ & Open &\\
         & Randomized & \cite{BDD11,Shah_Hermitian,banks2020pseudospectral,definite_dnc} & $O(n^{\omega_0}\log^2(n))$ & $O(n^3/\sqrt{M})$ & \makecell{
         Comm.\ optimal \\ for general matrices
         } \\
         \midrule
         \multirow{3}{*}{Jacobi} & Scalar & Section~\ref{section: classical} & $O(n^3)$ & $O(n^4/M)$ & $\sqrt{M}\leq n<M$ \\
         & Blocked & Section~\ref{section: block} & $O(n^2b + n^3b^{\omega_0-3})$ & $O(n^3/b)$ & $b = O(\sqrt{M})$\\
         & Recursive & Section~\ref{section: recursive} & $O(n^{3(1-f)+\omega_0f})$ & $O(\frac{n^{3(1-f) +\omega_0f}}{M^{\omega_0/2 - 1}})$ & $0 < f < 1$\\
         \midrule
         \multirow{1}{*}{Other} & Yau \& Lu & \cite{Yau_Lu}  & $O(n^{\omega_0} \log(n))$ & $O(\frac{n^{\omega_0} \log(n)}{M^{\omega_0/2 -1}})$ & Computes $e^{iA}$ \\
         \bottomrule
    \end{tabular}
    \caption{Complexity upper bounds for various (serial) symmetric eigensolvers.}
    \label{tab: complexity_bounds}
\end{table}
\renewcommand{\arraystretch}{1}

\indent We now summarize the complexity bounds derived in the remainder of the paper. \textcolor{black}{Each one  applies to a single \textit{sweep} of Jacobi's method -- that is, one pass over all off-diagonal entries of $\mathbf{A}$. Since the versions of Jacobi considered in this paper are guaranteed to converge (in a way made more precise later), these single-sweep results immediately yield complexity bounds for the full Jacobi procedure assuming convergence occurs after $O(1)$ sweeps.} This is a mild assumption, which we verify experimentally in Section~\ref{section: numerical_ex}. For a comparison with other solvers for the symmetric eigenproblem, see Table~\ref{tab: complexity_bounds}. \\
\indent To the best of our knowledge, existing Jacobi algorithms either attain \textcolor{black}{arithmetic or} communication \textcolor{black}{complexity} lower bounds without a convergence guarantee \cite{DrmacVeselic08a,DrmacVeselic08b}, 
or come with a convergence guarantee at a cost of sub-optimal complexity \cite{Drmac2009}. In this paper, we attain the best of both worlds, specifically by introducing tree-like recursion into the Jacobi algorithm. 
Our contributions can be described as follows: 
\begin{enumerate}
    \item In Section~\ref{section: classical} we consider Jacobi's method as originally introduced, which requires $\Theta(n^3)$ arithmetic operations and can be implemented with communication cost $O(n^4/M)$, assuming $\sqrt{M}\leq n < M$. The latter is provably optimal (see Theorem~\ref{thm: classical_lower_bound}).
    \item In Section~\ref{section: block}, we demonstrate that by blocking Jacobi's method we can decrease the arithmetic complexity to $\Theta(n^2b + n^3b^{\omega_0-3})$ for a parameter $0 < b < n/2$ that defines the block size. When $b = \Theta(\sqrt{M})$, this version of Jacobi hits the communication lower bound for $O(n^3)$ matrix multiplication (see Theorem~\ref{thm: blocked_complexity}). 
    \item Section~\ref{section: recursive} presents a cache-oblivious, recursive version of the blocked algorithm, which is capable of getting essentially optimal complexity, both in terms of arithmetic operations and serial communication. In particular, we derive the respective bounds $O(n^{3(1-f)+\omega_0f})$ and $O(\frac{n^{3(1-f)+\omega_0f}}{M^{\omega_0/2-1}})$, where $0 < f < 1$ is the \textit{log block size}. Taking $f$ arbitrarily close to one yields near-optimal complexity $O(n^{\omega _0})$, albeit for ``galactically''  large $n$. 
    \item In Section~\ref{section: svd} we discuss one-sided Jacobi SVD algorithms, demonstrating that essentially all of our complexity bounds extend to that setting as well. 
\end{enumerate}

Each of Sections~\ref{section: classical}--\ref{section: svd} is structured as follows:\ after introducing a version of Jacobi's method (or Jacobi SVD), we discuss convergence conditions and prove formal complexity bounds. We follow that in Section~\ref{section: numerical_ex} with a handful of numerical examples. Analogous complexity bounds for parallel versions of Jacobi's method will be presented in a forthcoming ``part two" to this work. \textcolor{black}{Extending the improved accuracy of standard Jacobi -- as established in \cite{DemmelVeselic92} -- to the versions of the algorithm presented here (and in part two) is an open problem left to future work}.

\subsection{Notation and Conventions}
Throughout the paper, matrices are represented with bold letters, with $\mathbf{A} \in {\mathbb R}^{n \times n}$ assumed to be symmetric. $|| \cdot ||_2$ and $|| \cdot||_F$ are the spectral and Frobenius norms, respectively, with $\kappa_2(\mathbf{A}) = ||\mathbf{A}||_2 || \mathbf{A}^{-1}||_2$ the spectral norm condition number. Additionally, $\sigma_{\min}(\mathbf{A})$ is the smallest singular value of $\mathbf{A}$, while $\mathbf{A}^{\dagger}$ denotes the Moore-Penrose pseudoinverse. \textcolor{black}{All complexity bounds are stated in standard big-O notation -- i.e., if $f(n)$ and $g(n)$ are positive functions of $n$ then $f(n) = O(g(n))$ if there exists a constant $C > 0$ such that $f(n) \leq Cg(n)$ for all $n$ sufficiently large. Similarly, $f(n) = \Omega(g(n))$ if $C'g(n) \leq f(n)$, again for $n$ sufficiently large and $C' > 0$ a constant. We write $f(n) = \Theta(g(n))$ if both $f(n) = O(g(n))$ and $f(n) = \Omega(g(n))$ hold.}

\section{Classical Jacobi}\label{section: classical}
We begin by considering the classical (i.e., scalar) version of Jacobi's method, originally derived in \cite{Jacobi1846} and presented below as Algorithm~\ref{alg:Jacobi1}. 

\begin{algorithm}
\caption{Classical Jacobi for the Symmetric Eigenproblem}
\label{alg:Jacobi1} \begin{algorithmic}[1]
\Require $\M{A} \in {\mathbb R}^{n \times n}$ is symmetric
\Ensure On output, $\M{A}$
is an (approximately) diagonal matrix $\M{D}$ containing the eigenvalues
of $\M{A}$ and $\M{Q}$ is an orthogonal matrix of (approximate) eigenvectors satisfying $\M{A} = \M{Q}\M{D}\M{Q}^T$.
\algrule
\Function{$[\M{Q},\M{A}]=$ Jacobi}{$\M{A}$} \State $\M{Q}=\M{I}$ \Repeat \For{all off-diagonal entries $(i,j)$
of $\M{A}$, $i<j$, in some order} \If{$|\M{A}(i,j)|$ is large
enough} \State $\M{\hat{A}}=\M{A}([i,j],[i,j])$ \Comment{$\M{\hat{A}}$
is the $2 \times 2$ submatrix of $\M{A}$ in rows/columns $i$ and $j$}
\State Let $\M{\hat{A}}=\M{\hat{Q}}\M{\hat{D}}\M{\hat{Q}}^{T}$ be
the eigendecomposition of $\M{\hat{A}}$ \State Multiply rows $i$
and $j$ of $\M{A}$ by $\M{\hat{Q}}^{T}$ and columns $i$ and $j$
by $\M{\hat{Q}}$ \State Multiply columns
$i$ and $j$ of $\M{Q}$ by $\M{\hat{Q}}$ \EndIf \EndFor \Until{all
off-diagonal entries of $\M{A}$ are small enough} \EndFunction \end{algorithmic} 
\end{algorithm}

The heuristic behind this algorithm is fairly simple. Line~8 performs an orthogonal similarity transformation on $\mathbf{A}$, preserving its eigenvalues while zeroing out entries $\mathbf{A}(i,j)$ and $\mathbf{A}(j,i)$ simultaneously, thereby reducing the squared Frobenius norm of the matrix of off-diagonal entries of $\mathbf{A}$ by $2|\mathbf{A}(i,j)|^{2}$. Once these off-diagonal entries are sufficiently small, so that $\mathbf{A}$ approximately equals a diagonal matrix $\mathbf{D}$, its diagonal entries can be taken as eigenvalue approximations. Moreover, line~9 maintains the orthogonality of $\mathbf{Q}$ and guarantees that the outputs satisfy $\mathbf{A}=\mathbf{Q}\mathbf{D}\mathbf{Q}^{T}$ -- i.e., we obtain an approximate eigendecomposition of $\mathbf{A}$. \\
\indent In this approach, convergence to a diagonal matrix is typically quantified via the off-diagonal ``norm"
\begin{equation}\label{eqn: Omega_A}
    \Xi(\mathbf{A})\coloneqq\sqrt{\sum_{i\neq j}\left|\mathbf{A}(i,j)\right|^{2}}.
\end{equation}
As mentioned above, the inner-most loop of Algorithm~\ref{alg:Jacobi1} reduces \textcolor{black}{$\Xi(\mathbf{A})^2$} by $2|\mathbf{A}(i,j)|^2$ for each $(i,j)$ selected. The rate by which $\Xi(\mathbf{A})$ decreases is usually measured against the number of \textit{sweeps} performed by the algorithm, where one sweep corresponds to a pass over all off-diagonal entries of $\mathbf{A}$ (equivalently, one execution of the loop beginning at Line~3).\footnote{Sweeps are sometimes called ``segments" in earlier literature -- e.g., \cite{henrici1958speed}.} 
Importantly, a guarantee that $\Xi(\mathbf{A})$ converges to zero does not imply that $\mathbf{A}$ approaches a \textit{fixed} diagonal matrix, only that it becomes diagonal as the number of sweeps increases.

The matrix $\M{\hat{Q}}$ in line~7 can be obtained cheaply, either by calling the specialized routine \texttt{slaev2} in LAPACK \cite{LAPACK_User_Guide} or, as in Jacobi's original work, by setting
\begin{equation}\label{eqn: rotaton_matrix}
    \M{\hat{Q}} = \begin{pmatrix} \cos(\theta) & -\sin(\theta) \\ \sin(\theta) & \cos(\theta)   
    \end{pmatrix}
\end{equation}
for an angle $\theta$ satisfying
\begin{equation}\label{eqn: roation_angle}
    \cot(2\theta) = \frac{\mathbf{A}(i,i)-\mathbf{A}(j,j)}{2\mathbf{A}(i,j)} = \frac{\mathbf{\hat{A}}(1,1)-\mathbf{\hat{A}}(2,2)}{2\mathbf{\hat{A}}(1,2)}.
\end{equation}
Other aspects of Algorithm~\ref{alg:Jacobi1} can be varied as well, including the order in which off-diagonal entries are zeroed out and how ``large enough'' and ``small enough'' are defined in lines~5 and 12, respectively. There is a significant amount of work in the literature discussing the impact
of these choices on convergence rate \cite{Forsythe_Henrici_1960,Mascarenhas_1995} and final accuracy \cite{DemmelVeselic92,DrmacVeselic08a,DrmacVeselic08b,DGESVD99}. \\
\indent The simplest way to guarantee convergence for Algorithm~\ref{alg:Jacobi1}, indeed the way it was originally posed in \cite{Jacobi1846}, is to zero out the largest off-diagonal entry of $\mathbf{A}$ at each step. In practice, a max-off-diagonal ordering  is usually prohibitively expensive, as the largest entry would (necessarily) need to be updated after each iteration. 
Instead, cyclic orderings -- which can attain the same convergence rate without tracking a largest entry (see the references below) -- are the standard. As an example, one we will return to throughout, the \textit{column-cyclic} ordering cycles (within each sweep) through above-diagonal indices $(i_k,j_k)$ with $i_k < j_k$ according to
\begin{equation}\label{eqn: column_cyclic}
    (i_{k+1},j_{k+1}) = \begin{cases}
        (i_k+1,j_k) & i_k<j_k-1, j_k \leq n \\
        (1,j_k+1) & i_k = j_k-1, j_k \leq n-1 \\
        (1,2) & i_k = n-1, j_k = n
    \end{cases} \; \; \; \;  \; \; k = 0,1,\ldots
\end{equation}
beginning with $(i_0,j_0) = (1,2)$. For a discussion of alternative cyclic orderings see \cite{Shroff_Schreiber_1989}. \textcolor{black}{Importantly, convergence with respect to \eqref{eqn: column_cyclic} immediately implies convergence with respect to any other equivalent cyclic ordering (see \cite[Remark 2.3]{Drmac2009}).} \\
\indent Rigorous convergence results for so called cyclic Jacobi algorithms originate with Forsythe and Henrici \cite{Forsythe_Henrici_1960,henrici1958speed}, who demonstrated that a version of Algorithm~\ref{alg:Jacobi1} that chooses off-diagonal entries according to \eqref{eqn: column_cyclic} and computes $\mathbf{\hat{Q}}$ via \eqref{eqn: rotaton_matrix} will converge -- even omitting the check in line~5 -- provided the rotation angles $\theta$ are bounded away from $\pi/2$. In practice, failure is quite rare even when this requirement is not satisfied.\footnote{Forsythe and Henrici provide a handful of $3 \times 3$ failure examples, though they are fairly contrived. Larger examples for general cyclic orderings can be found in \cite{Hansen}.} In fact, simply selecting off-diagonal entries randomly is sufficient to guarantee convergence with high probability, regardless of how $\mathbf{\hat{Q}}$ is constructed, as shown below. 

\begin{prop} \label{prop: rand_converges}
    If in line~4 of Algorithm~\ref{alg:Jacobi1} off-diagonal entries are selected uniformly at random (with or without replacement) then $\Xi(\mathbf{A}) \rightarrow 0$ almost surely as the number of sweeps goes to infinity.
\end{prop}
\begin{proof} Let $L$ be the limit of $\Xi(\mathbf{A})$ as the number of sweeps goes to infinity.\footnote{Note that this limit always exists since $\Xi(\mathbf{A})$ is non-increasing and bounded from below.} Since $\Xi(\mathbf{A})$ cannot increase, we always have $\Xi(\mathbf{A}) \geq L$ . Suppose now $L \neq 0$. In this case, 
    \begin{equation}\label{eqn: omega_limit_condition}
        \Xi(\mathbf{A}) - L < \frac{2L^2}{n(n-1)}
    \end{equation}
    beyond a certain point in the algorithm. At the start of each sweep beyond this point, $\mathbf{A}$ is guaranteed to have an off-diagonal entry with $\mathbf{A}(i,j)^2 \geq \frac{L^2}{n(n-1)}$, a consequence of the lower bound $\Xi(\mathbf{A}) \geq L$. Selecting such an entry as the first to be zeroed out, which occurs with probability at least $2/(n(n-1))$, will cause $\Xi(\mathbf{A})$ to fall below $L$. This implies a contradiction almost surely as the number of sweeps goes to infinity. 
\end{proof} 

\vspace{3mm}
\indent We turn now to complexity bounds for Algorithm~\ref{alg:Jacobi1}. First, we note that one sweep through all the off-diagonal entries of $\mathbf{A}$ costs $\Theta(n^{3})$ flops. 
We next derive a lower bound on the serial communication cost, assuming that the only freedom we have is the order in which off-diagonal entries of $\mathbf{A}$ are chosen in line~4 (in particular leaving blocking/recursion to the subsequent sections). Theorem~\ref{thm: classical_lower_bound} implies that, while we can avoid some communication, we cannot attain our goal of $\Omega(n^{3}/\sqrt{M})$. Its proof applies an argument analogous to, but simpler than, the one used to analyze matrix multiplication in \cite{BDHS11}. 

\begin{theorem}\label{thm: classical_lower_bound}
    Let $W(n)$ denote the serial communication complexity of \textcolor{black}{one sweep of} Algorithm~\ref{alg:Jacobi1}. If $M$ is the size of available fast memory and $\sqrt{M} \leq n < M$ 
    then $W(n) = \Omega(n^4/M)$. This lower bound can be attained with a specific choice of ordering in line 4.
\end{theorem}

\begin{proof} We start by observing that one execution of lines~6-8 does $\Theta(n)$ flops on $2n$ entries of $\mathbf{A}$ (taking advantage of symmetry). Next, we note that $k\leq n$ consecutive executions of lines~6-8 must necessarily access at least $\sqrt{k}$ different rows and/or columns of $\mathbf{A}$. This fact is based on the simple geometric observation (simpler than Loomis-Whitney in  \cite{BDHS11}) that $k$ different lattice points $(i,j)$ must have at least $\sqrt{k}$ different values of $i$ and $j$.\footnote{If $n_{i}$ is the number of different values of $i$, and $n_{j}$ the number of different values of $j$, we want to minimize $\max(n_{i},n_{j})$ subject to $n_{i}\cdot n_{j}\geq k$.} In other words, as long as we are restricted to changing the order in which off-diagonal entries of $\mathbf{A}$ are zeroed out, and $M>n$ so we can fit $\sqrt{k}=\lceil M/n \rceil$ rows or columns of $\mathbf{A}$ into fast memory, we can perform $O(kn)=O(M^{2}/n)$ flops on this submatrix data. Following the proof of \cite[Theorem 2.2]{BDHS11}, we can divide the instruction stream executing Algorithm~\ref{alg:Jacobi1} into {\em segments} each containing $M$ reads and writes between fast and slow memory, during which $O(M)$ entries of $\mathbf{A}$ are available for execution. Since we need to do $\Theta(n^{3})$ flops per sweep, this requires $\Omega(n^{3}/(M^{2}/n))=\Omega(n^{4}/M^{2})$ segments, or $\Omega(n^{4}/M^{2}\cdot M)=\Omega(n^{4}/M)$ reads and writes in total. \\
    \indent This ``restricted'' lower bound is attainable by updating all the $(i,j)$ pairs in an appropriate ``blocked'' order. Specifically,  take the upper triangle of $(i,j)$ pairs with $1\leq i<j\leq n$ and break it into
    squares\footnote{For diagonal blocks, we assume that they take upper triangular forms -- i.e., we have triangles on the diagonal.}  of size $O(M/n) \times O(M/n)$. 
    In this case, we can execute each block completely before handling the next one. This requires $O(M)$ memory by construction (the cost of reading in the corresponding rows/columns) letting us do $O((M/n)^{2}n)=O(M^{2}/n)$ flops on $O(M)$ data, attaining the communication lower bound.
\end{proof} 

\vspace{3mm}

Note that this bound exceeds the conjectured $\Omega(n^{3}/\sqrt{M})$ as long as $\sqrt{M} \leq n$ -- i.e., when the entire matrix does not fit in fast memory. Nevertheless, when $\sqrt{M} \leq  n <M $, the communication-optimal version of Algorithm~\ref{alg:Jacobi1}, guaranteed by Theorem~\ref{thm: classical_lower_bound}, improves on a naive implementation, which communicates $O(n^{3})$ words. When $M<n$, so not even one row or column fits in fast memory, both versions communicate $O(n^{3})$ words.

\section{Block Jacobi}\label{section: block}

The first tool at our disposal to improve these complexity bounds is blocking -- i.e., modifying Algorithm~\ref{alg:Jacobi1} so that $b \times b$ off-diagonal \textit{blocks} are zeroed out for some $2b < n$. The innermost loop of such a blocked Jacobi algorithm, presented here as Algorithm~\ref{alg:Jacobi2}, diagonalizes a $2b \times 2b$ symmetric matrix at each step. In this context, each sweep corresponds to a pass over all off-diagonal blocks of $\mathbf{A}$, with convergence again measured by $\Xi(\mathbf{A})$. Like the classical algorithm, blocked versions of Jacobi's method have a rich history in the literature (see e.g., \cite{Shroff_Schreiber_1989,Drmac2009,Hari_2015,VL_Jacobi}).

\begin{algorithm}[h]
\caption{Block Jacobi for the Symmetric Eigenproblem}
\label{alg:Jacobi2} \begin{algorithmic}[1] \Require $\mathbf{A} \in {\mathbb R}^{n \times n}$ is symmetric and partitioned as follows (for block size $b$ that divides $n$):
\[
\mathbf{A}_{IJ}=\mathbf{A}((I-1)b+1 : Ib, (J-1)b+1:Jb)\; \; \; \; \text{for} \; \; \; \; 1 \leq I,J \leq n/b.
\]
\Ensure On output, $\mathbf{A}$ is an (approximately) diagonal matrix $\mathbf{D}$ containing the eigenvalues of $\mathbf{A}$ and $\mathbf{Q}$ is an orthogonal matrix of (approximate) eigenvectors satisfying $\mathbf{A} = \mathbf{Q}\mathbf{D}\mathbf{Q}^{T}$. 
\algrule
\setstretch{1.1}
\Function{$[\mathbf{Q},\mathbf{A}]=$
Block\_Jacobi}{$\mathbf{A}$} \State $\mathbf{Q}=\mathbf{I}$
\Repeat \For{all off-diagonal blocks $(I,J)$ of $\mathbf{A}$, $I<J$,
in some order,} \State $\mathbf{\hat{A}}=\mathbf{A}([I,J],[I,J])$ \Comment{$\mathbf{\hat{A}}=$ $2b \times 2b$ submatrix of $\mathbf{A}$ in
block rows/cols $I$ and $J$} \If{$\mathbf{\hat{A}}$ is far
enough from diagonal} \State Let $\mathbf{\hat{A}}=\mathbf{\hat{Q}}\mathbf{\hat{D}}\mathbf{\hat{Q}}^{T}$
be an 
(approximate) eigendecomposition of $\mathbf{\hat{A}}$ 
\State $\mathbf{\hat{Q}}_1 = \mathbf{\hat{Q}}(1:b,1:2b)$
\State $[\mathbf{P},\mathbf{L},\mathbf{U}] =$ \textsc{Recursive}\_LUPP$(\mathbf{\hat{Q}}_1^T)$
\rlap{\smash{$\left.\begin{array}{@{}c@{}}\\{}\\{}\end{array}\color{black}\right\}%
          \color{black}\begin{tabular}{l}Optional\end{tabular}$}}
\State $\mathbf{\hat{Q}} = \mathbf{\hat{Q}}\mathbf{P}^T$
\State Multiply block rows $I$ and $J$ of $\mathbf{A}$ by $\mathbf{\hat{Q}}^{T}$
\State Multiply block columns $I$ and $J$ of $\mathbf{A}$ by $\mathbf{\hat{Q}}$
\State Multiply block columns $I$ and
$J$ of $\mathbf{Q}$ by $\mathbf{\hat{Q}}$ \EndIf \EndFor \Until{all off-diagonal
entries of $\mathbf{A}$ are small enough} \EndFunction \end{algorithmic} 
\end{algorithm}

\begin{remark}\label{rem: approx_diag}
   \normalfont{In the blocked setting, fully diagonalizing each submatrix $\mathbf{\hat{A}}$ is not strictly required to reduce $\Xi(\mathbf{A})$. Saad \cite{saad2023revisiting}, for example, states a version of Algorithm \ref{alg:Jacobi2} that \textit{block} diagonalizes $\mathbf{\hat{A}}$ at each step. The upshot of this approach is that the corresponding $\mathbf{\hat{Q}}$ -- obtained by solving a Riccati-type equation -- exhibits additional structure, which can be leveraged to more efficiently handle the subsequent block multiplications. An even coarser \textit{approximate} diagonalization, e.g., obtained from a single sweep of Algorithm~\ref{alg:Jacobi1}, can also be used in line~7, though this may affect convergence rate.}
\end{remark}

\indent As in scalar Jacobi, the order in which off-diagonal blocks are chosen in line~4 may be done dynamically (e.g., by selecting the block with the largest Frobenius norm), cyclically (extending \eqref{eqn: column_cyclic} to block indices), or randomly. Cyclic orderings are again the most popular choice \cite{Drmac2009,foulser1989blocked, Shroff_Schreiber_1989}, in part due to their simplicity but also, as in the proof of Theorem~\ref{thm: classical_lower_bound}, because they can reduce communication. \\
\indent Regardless of the ordering used, we again need to be mindful of convergence, as a naive implementation of blocked Jacobi -- like its scalar counterpart -- may fail (in the sense that $\Xi(\mathbf{A})$ may not converge to zero). The first global convergence results for blocked Jacobi were derived by Drmač \cite{Drmac2009}, who showed $\Xi(\mathbf{A}) \rightarrow 0$ under two conditions:
\begin{enumerate}
    \item Off-diagonal blocks are chosen in a column/row-cyclic fashion.
    \item The smallest singular value of the upper left $b \times b$ block of each $2b \times 2b$ rotation matrix $\mathbf{\hat{Q}}$ is bounded away from zero. 
\end{enumerate}
In essence, these conditions generalize the convergence requirement for the scalar algorithm derived by Forsythe and Henrici \cite{Forsythe_Henrici_1960}. \\
\indent \textcolor{black}{To attain the singular value bound necessary for convergence, we permute the columns of $\mathbf{\hat{Q}}$ in line 10 of Algorithm~\ref{alg:Jacobi2}. The corresponding permutation matrix is obtained from a PLU factorization $\mathbf{\hat{Q}_1}^T = \mathbf{P}^T \mathbf{L} \mathbf{U}$, where $\mathbf{\hat{Q}_1} \in {\mathbb R}^{b \times 2b}$ contains the first $b$ rows of the block rotation $\mathbf{\hat{Q}}$, $\mathbf{L} \in {\mathbb R}^{2b \times b}$ is unit lower trapezoidal, and $\mathbf{U} \in {\mathbb R}^{b \times b} $ is upper triangular. When this LU decomposition is obtained with partial pivoting (referred to here as LUPP), the sub-diagonal ($i > j$) entries of $\mathbf{L}$ satisfy $|\mathbf{L}_{ij}| \leq 1$. This observation alone implies a singular value bound on the upper left $b \times b$ block of $\mathbf{\hat{Q}} \mathbf{P}^T$, which is stated below as Lemma \ref{lemma: singular_value_bound}. Since Drmač's proof of convergence for blocked, cyclic Jacobi \cite[Theorem 2.7]{Drmac2009} depends only on the existence of a lower bound like this one, Lemma \ref{lemma: singular_value_bound} immediately guarantees convergence for a version of Algorithm~\ref{alg:Jacobi2} that executes lines~8-10 and runs through off-diagonal blocks according to a column-cyclic ordering \eqref{eqn: column_cyclic}. We also note that this lemma easily extends to unitary $\mathbf{Q}$, thereby implying convergence for an analogous version of blocked Jacobi for Hermitian matrices. }

\begin{lemma}\label{lemma: singular_value_bound}
     Let $\M{Q} \in {\mathbb R}^{2b \times 2b}$ be an orthogonal matrix with leading rows $\M{Q}_1 \in {\mathbb R}^{b \times 2b}$. Suppose that $\M{Q}_1^T = \M{P}^T\M{L}\M{U}$ is a PLU decomposition computed with partial pivoting. If $\widetilde{\M{Q}}_{11}$ is the upper left $b \times b$ block of $\M{Q}\M{P}^T$ then
     $$ \sigma_{\min}(\mathbf{\widetilde{Q}}_{11}) \geq 3\sqrt{2} \left[ (3b^2 + b)(4^b + 6b - 1) \right]^{-1/2}. $$
\end{lemma} 
\begin{proof} If $\mathbf{L}_1 \in {\mathbb R}^{b \times b}$ is the upper block of $\mathbf{L}$, then $\mathbf{\widetilde{Q}}_{11} = \mathbf{U}^T\mathbf{L}_1^T$ and therefore $\sigma_{\min}(\mathbf{\widetilde{Q}}_{11}) \geq \sigma_{\min}(\mathbf{U}) \sigma_{\min}(\mathbf{L}_1)$.
    Since the columns of $\mathbf{Q}_1^T$ are orthonormal, the decomposition $\mathbf{Q}_1^T = \mathbf{P}^T\mathbf{L}\mathbf{U}$ implies $\mathbf{U}^{-1} = \mathbf{Q}_1\mathbf{P}^T\mathbf{L}$. Hence, $||\mathbf{U}^{-1}||_2 = ||\mathbf{Q}_1\mathbf{P}^T\mathbf{L}||_2 \leq ||\mathbf{L}||_2 $
    and we have 
    \begin{equation}
        \sigma_{\min}(\mathbf{\widetilde{Q}}_{11}) \geq ||\mathbf{L}||_2^{-1} \sigma_{\min}(\mathbf{L}_1).
    \end{equation}
    The bound now follows from $||\mathbf{L}||_2 \leq \frac{1}{\sqrt{2}} (3b^2 + b)^{1/2}$ and $||\mathbf{L}_1^{-1}||_2 \leq \frac{1}{3} (4^b + 6b-1)^{1/2}$, both of which can be obtained by passing to the Frobenius norm and using the fact that $\mathbf{L}_1$ and $\mathbf{L}$ are unit lower triangular/trapezoidal with $|\mathbf{L}_{ij}| \leq1$ for $i > j$. The bound on $||\mathbf{L}_1^{-1}||_2$ requires computing the Frobenius norm of a matrix whose entries are powers of two; see \cite[Chapter 6]{Solving_LS} for the details.  
\end{proof}

\vspace{3mm}
\indent \textcolor{black}{This is the first work to propose using LUPP to guarantee convergence in blocked Jacobi. In short, this is due to our focus on asymptotic complexity, as LUPP can be done with only $O(n^{\omega_0})$ flops if implemented recursively (see \cite{Toledo97} and \cite[Section 4.2]{demmel2007fast}). This is in contrast to alternatives like column-pivoted QR -- the choice of Drmač in \cite{Drmac2009} -- which are much more expensive. We discuss this further  in Remark~\ref{rem: QRCP} at the end of this section. For completeness, we also provide pseudocode for the version of LUPP that achieves optimal arithmetic complexity in Appendix~\ref{section: LUPP}.}\\
\indent As in the scalar case, failure modes for blocked Jacobi are rare. Consequently, it is typically not worth executing lines~8-10 of Algorithm~\ref{alg:Jacobi2} on every submatrix. Instead, LUPP should be invoked only if it appears that $\Xi(\mathbf{A})$ is stagnating. Alternatively, extending Proposition~\ref{prop: rand_converges} to the blocked setting, simply selecting off-diagonal blocks in random order should guarantee convergence with high probability. \\
We are now ready to derive complexity bounds for Algorithm~\ref{alg:Jacobi2}, all of which apply whether or not the call to recursive LUPP is executed. We assume here that $b = O(\sqrt{M})$ -- i.e., the usual ``block fits in cache" condition. 
\begin{theorem}\label{thm: blocked_complexity}
    Let $F(n)$ and $W(n)$ denote, respectively, the arithmetic and (serial) communication complexities of \textcolor{black}{one sweep of}  Algorithm~\ref{alg:Jacobi2} under the following conditions:
    \begin{enumerate}
        \item $O(n^{\omega_0})$ matrix multiplication is used.
        \item $b = O(\sqrt{M})$ is chosen so that lines 7-10 can take place entirely in fast memory of size $M$.
    \end{enumerate}
    Then $F(n) = \Theta(n^2b + n^3b^{\omega_0-3})$ and $W(n) = O(n^3/b)$. In particular, $W(n) = O(n^3/\sqrt{M})$ if $b = \Theta(\sqrt{M})$.
\end{theorem}
\begin{proof} We start with a straightforward flop count. For each block pair $(I,J)$, Algorithm~\ref{alg:Jacobi2} requires $\Theta(b^3)$ flops to compute $\mathbf{\hat{Q}}$ (including the optional call to LUPP) and an additional $\Theta(nb^{\omega_0-1})$ flops for the subsequent block row/column multiplications. Since there are $\Theta(n^2/b^2)$ block pairs in total, we conclude that the arithmetic cost of one sweep of Algorithm~\ref{alg:Jacobi2} is $\Theta(n^2b+n^3b^{\omega_0-3})$. \\
    \indent For communication, we note that line~5 requires reading in $O(b^2)$ entries of $\mathbf{A}$, after which the subsequent diagonalization and (optional) LUPP factorization can be done without any additional communication (thanks to our choice of $b$). The final block multiplications in lines 11-13 can then be executed by reading in one $2b \times 2b$ block of $\mathbf{A}$ or $\mathbf{Q}$ at a time, which amounts to $O( \lceil \frac{n}{2b} \rceil (2b)^2) = O(nb)$ total reads. Hence, each sweep of Algorithm~\ref{alg:Jacobi2} has communication cost
    \begin{equation}\label{eqn: one_sweep_comm_block}
        O\left(\left\lceil \frac{n}{b} \right\rceil^2 \left(b^2 + nb \right) \right) = O(n^3/b),
    \end{equation}
    which completes the proof.
\end{proof} 

\vspace{3mm}

Note that our communication bound for Algorithm~\ref{alg:Jacobi2} does not depend on $\omega_0$. This is consequence of the fact that, when $b = O(\sqrt{M})$, communication is dominated by the cost of reading in blocks of $\mathbf{A}$ and $\mathbf{Q}$ in lines 11-13, which is independent of the matrix multiplication routine used to execute them. The arithmetic complexity, on the other hand, \textit{does} depend on $\omega_0$, though $F(n)$ can only be sub-$O(n^3)$ if $b$ is a small power of $n$. Even then, the bound from Theorem~\ref{thm: blocked_complexity} can never beat $\Theta(n^{2.5})$ for any $2 < \omega_0 < 3$,
and of course allowing $b$ to scale with $n$ will eventually run up against the assumption $b = O(\sqrt{M})$. In this way, Algorithm~\ref{alg:Jacobi2} cannot reach optimal complexity, though it does attain the $O(n^3/\sqrt{M})$ 
communication lower bound when $b = \Theta(\sqrt{M})$.

\begin{remark}\label{rem: QRCP}
    \normalfont{\textcolor{black}{Following Drmač \cite{Drmac2009}, QRCP can be used in place of LUPP in Algorithm~\ref{alg:Jacobi2} as follows:\ (1) compute a column-pivoted QR decomposition of $\mathbf{\hat{Q}_1}$ in line 9 and (2) allow the resulting permutation matrix to take the place of $\mathbf{P}^T$ in line 10. Drmač provides a counterpart to Lemma~\ref{lemma: singular_value_bound} for this procedure \cite[Lemma 2.2]{Drmac2009}, which implies convergence for versions of block Jacobi that incorporate it. Nevertheless, QRCP is expensive. In contrast to LUPP, it cannot be implemented with $O(n^{\omega_0})$ arithmetic complexity and is typically communication intensive (with pivots chosen by column norm). Communication-avoidant versions can attain the communication lower bound of $O(n^3)$ matrix multiplication \cite{CARRQR,duersch2020randomized} -- which will be relevant for the next section -- but they are nontrivial to implement and may require additional flops. Reducing \textit{both} arithmetic and communication in QRCP typically requires limiting the number of columns it considers \cite{fakih2025,armstrong2025,QDEIM}, though such algorithms are primarily relevant for matrices much wider than $b \times 2b$. Finally, we note that while stronger rank-revealing guarantees for QRCP-based algorithms (for example \cite{grigori2025}) imply tighter singular value bounds than Lemma~\ref{lemma: singular_value_bound}, they have no bearing on single-sweep complexity.} \\
    \indent Alternatives to deterministic pivoting methods like QRCP/LUPP include algorithms that randomly sample the columns of $\mathbf{\hat{Q}} (1:b,1:2b)$ -- e.g;., via leverage score sampling \cite{ma2014statistical} or a determinental point process \cite{kulesza2011k} -- and re-organize $\mathbf{\hat{Q}}$ accordingly, that is, by shuffling the corresponding columns of $\mathbf{\hat{Q}}$ to the ``front" of the matrix. While there are promising theoretical results for these sampling techniques (see for example \cite{DMM06} \cite[Chapter 6]{RandLAPACK} \cite{poulson2020high}) neither comes with a guarantee of convergence and may be even more costly than QRCP.}
\end{remark}
\section{Recursive Jacobi}\label{section: recursive}
As a final push toward optimal complexity, we consider in this section a recursive version of Jacobi's method (Algorithm~\ref{alg:Jacobi3}). Here, we allow blocked Jacobi to call itself recursively on each block subproblem $\mathbf{\hat{A}}$, only defaulting to direct diagonalization once the problem size falls below a certain threshold, which depends on the fast memory size $M$. This approach is rooted in the observations from the previous section. We saw in the standard blocked case that the cumulative cost of block multiplications/LUPP factorizations could be pushed closer to $O(n^{\omega_0})$ by increasing the block size; by calling Algorithm~\ref{alg:Jacobi2} recursively, we can reap this benefit while eschewing the need to diagonalize a large matrix directly. For simplicity, the blocking strategy in Algorithm~\ref{alg:Jacobi3} is also defined recursively, at each level setting the block size $b = n^f$ 
for $0<f<1$ the log block size, which remains constant through the recursion. In practice, holding $f$ constant is not strictly required, as we discuss in more detail below. \\
\indent Once again, LUPP is included in Algorithm~\ref{alg:Jacobi3} to guarantee convergence. This can be shown easily by bootstrapping the proof of convergence in the blocked case (see Proposition~\ref{prop:recursive_convergence}). Note that the practical considerations from the previous section again apply here:\ failure cases are rare, so LUPP should realistically be employed only when $\Xi(\mathbf{A})$ fails to decrease sufficiently. 

\begin{algorithm}[t]
\caption{Recursive Jacobi for the Symmetric Eigenproblem}
\label{alg:Jacobi3} \begin{algorithmic}[1] \Require $\mathbf{A} \in {\mathbb R}^{n \times n}$ is symmetric and partitioned as follows for log block size $0 < f < 1$:
\[
\mathbf{A}_{IJ}=\mathbf{A}((I-1)b+1:Ib, (J-1)b+1:Jb) \; \; \; \; \text{for} \; \; \; \; b = n^f, \;  1 \leq I,J \leq n^{1-f}.
\]
\Ensure On output, $\mathbf{A}$ is an (approximately) diagonal matrix $\mathbf{D}$ containing the eigenvalues of $\mathbf{A}$ and $\mathbf{Q}$ is an orthogonal matrix of (approximate) eigenvectors satisfying $\mathbf{A} = \mathbf{Q}\mathbf{D}\mathbf{Q}^{T}$. 
\algrule
\setstretch{1.1}
\Function{$[\mathbf{Q},\mathbf{A}]=$ Recursive\_Jacobi}{$\mathbf{A}$,$f$}
\State
$\mathbf{Q}=\mathbf{I}$ \If{$n$ is small enough (less than $n_{\text{threshold}}$) or $2b \geq n$}
\State Solve directly: $\mathbf{A}=\mathbf{Q}\mathbf{D}\mathbf{Q}^{T}$; $\mathbf{A} = \mathbf{D}$ \Comment{Base case}
\Else \Repeat \For{all off-diagonal blocks $(I,J)$ of $\mathbf{A}$, $I<J$, in some order,} \State $\mathbf{\hat{A}}=\mathbf{A}([I,J],[I,J])$ \Comment{$\mathbf{\hat{A}}=$ the $2b \times 2b$ submatrix of
$\mathbf{A}$ in block rows/cols $I$ and $J$} \If{$\mathbf{\hat{A}}$
is far enough from diagonal} \State $[\mathbf{\hat{Q}},\mathbf{\hat{A}}]=$
\textsc{Recursive}\_\textsc{Jacobi}($\mathbf{\hat{A}}$,$f$) 
\State $\mathbf{\hat{Q}}_1 = \mathbf{\hat{Q}}(1:b,1:2b)$
\State $[\mathbf{P},\sim,\sim] =$ \textsc{Recursive}\_LUPP$(\mathbf{\hat{Q}}_1^T)$
\rlap{\smash{$\left.\begin{array}{@{}c@{}}\\{}\\{}\end{array}\color{black}\right\}%
          \color{black}\begin{tabular}{l}Optional\end{tabular}$}}
\State $\mathbf{\hat{Q}} = \mathbf{\hat{Q}}\mathbf{P}^T$
 \State Multiply block rows $I$ and $J$ of $\mathbf{A}$ by $\mathbf{\hat{Q}}^{T}$
\State Multiply block columns $I$ and $J$ of $\mathbf{A}$ by $\mathbf{\hat{Q}}$
\State Multiply block columns $I$ and $J$ of $\mathbf{Q}$ by $\mathbf{\hat{Q}}$
\EndIf \EndFor \Until{all off-diagonal entries of $\mathbf{A}$ are
small enough} \EndIf \EndFunction \end{algorithmic}
\end{algorithm}

\begin{prop}\label{prop:recursive_convergence} 
    If each sweep of Algorithm~\ref{alg:Jacobi3} works through off-diagonal blocks according to the column-cyclic ordering \eqref{eqn: column_cyclic} and executes the optional pivoting step (lines 11-13), 
    then it converges for any symmetric input matrix $\mathbf{A} \in {\mathbb R}^{n \times n}$. 
\end{prop}
\begin{proof} Let $k$ be the number of recursive steps performed by Algorithm~\ref{alg:Jacobi3}. We prove convergence for any value of $k$ inductively. The base case $k = 1$ corresponds to standard blocked Jacobi from Section~\ref{section: block}, and follows from \cite[Theorem 2.7]{Drmac2009} and Lemma~\ref{lemma: singular_value_bound}. Suppose now $k > 1$. At the highest level, Algorithm~\ref{alg:Jacobi3} calls itself on $2b \times 2b$ submatrices $\mathbf{\hat{A}}$. By our induction hypothesis we can assume that each of these calls is successful and in particular that $\Xi(\mathbf{\hat{A}})$ is reduced by a multiplicative factor $0 < \rho < 1$. Since Drmač's proof of convergence in the blocked case requires only this characterization of approximate diagonalization, we can simply repeat his argument to obtain $\Xi(\mathbf{A}) \rightarrow 0$, where executing lines 11-13 at the highest level ensures convergence via Lemma~\ref{lemma: singular_value_bound}.
\end{proof}
\vspace{3mm}

\indent We now present complexity bounds for Algorithm~\ref{alg:Jacobi3}, which -- as in the previous section -- apply whether or not the optional calls to (fast) LUPP are made. This result should be parsed carefully, as its conditions highlight some of the subtleties of constructing a recursive algorithm. First, we note that the log block size $f$ cannot be chosen completely arbitrarily; if $f$ is too close to one, the blocking parameter $b$ may be too large, prompting the algorithm to repeatedly call itself on the input matrix $\mathbf{A}$. Going further, even if the problem size shrinks initially it may stagnate before reaching the threshold for direct diagonalization. The second condition of Theorem~\ref{Thm:complexity_recursive_Jacobi} ensures that neither occurs. 

\begin{theorem}\label{Thm:complexity_recursive_Jacobi} 
Let $F(n)$ and $W(n)$ denote, respectively, the arithmetic and (serial) communication complexities of \textcolor{black}{one sweep of} Algorithm~\ref{alg:Jacobi3} under the following conditions:
\begin{enumerate}
    \item $O(n^{\omega_0})$ matrix multiplication is used.
    \item $0 < f < 1$ and $n_{\text{\normalfont threshold}}$ are constants satisfying $(1-f)^{-1} < \log_2(n_{\text{\normalfont threshold}})$ and $n_{\text{\normalfont threshold}} \leq \sqrt{M/2}$ for $M$ the size of available fast memory.
\end{enumerate} 
Then $F(n)=O\left( n^{3(1-f)+\omega_{0}f}\right)$ and $W(n) = O\left( \frac{n^{3(1-f)+\omega_0f}}{M^{\omega_0/2 - 1}} \right)$.
\end{theorem}
\begin{proof} We start with arithmetic complexity, noting the following for one sweep of Algorithm~\ref{alg:Jacobi3}:
\begin{enumerate}
\item Line~10 costs $F(2b)=F(2n^{f})$ flops by the inductive definition of $F$. 
\item Recalling the discussion in Section~\ref{section: block}, lines 11-13 require $O((2b)^{\omega_0}) = O(n^{\omega_0f})$ flops.
\item If we implement line 14 by doing $\left\lceil \frac{n}{2b}\right\rceil $ $2b \times 2b$ matrix multiplications, it will cost 
\begin{equation}\label{eqn: cost_of_block_mult}
    O\left(\left\lceil \frac{n}{2b}\right\rceil \cdot(2b)^{\omega_{0}}\right)=O\left(n^{1-f+\omega_{0}f}\right)
\end{equation}
flops. We can do the same in lines~15 and 16.
\item In total, lines~11-16 contribute $O\left(n^{3(1-f)+\omega_{0}f}\right)$ flops, as there are $O\left(\left\lceil \frac{n}{b}\right\rceil ^{2}\right) = O(n^{2-2f})$ blocks to iterate through, each of which requires $O\left(\left\lceil \frac{n}{b}\right\rceil \cdot(2b)^{\omega_{0}} + (2b)^{\omega_0}\right)$ operations. Hence the total cost is
\begin{equation}\label{eqn: total_non_recursive_flops}
    O\left(\left\lceil \frac{n}{b}\right\rceil ^{2}\right)\cdot O\left(\left\lceil \frac{n}{2b}\right\rceil \cdot(2b)^{\omega_{0}} + (2b)^{\omega_0}\right)=O\left(n^{3(1-f)+\omega_{0}f}\right).
\end{equation}
\item Combining the preceding items, the total arithmetic cost for each sweep (i.e., lines~6-19) is 
\begin{equation}\label{eqn:Jacobi1}
F(n)=O(n^{3(1-f)+\omega_{0}f})+O(n^{2-2f})F(2n^{f}).
\end{equation}
\item The exponent of the first term in \eqref{eqn:Jacobi1} satisfies $\omega_{0}=\omega_{0}(1-f)+\omega_{0}f < 3(1-f)+\omega_{0}f \leq 3.$ To approach $O(n^{\omega_{0}})$ we therefore want $f \rightarrow 1$. At the same time $f$ cannot be so close to one that $2n^f$ is actually larger than $n$. In particular we need $n^{1-f} > 2$, which is implied by $(1-f)^{-1} < \log_2(n_{\text{threshold}})$ since $n_{\text{threshold}}\ll n$.
\end{enumerate}

We now sum \eqref{eqn:Jacobi1} recursively. Moving one level deeper yields
\begin{equation}\label{eqn: sum_recursively_one}
    \aligned
    F(n) &= O(n^{3(1-f)+\omega_0f}) + O(n^{2-2f})\left[O((2n^f)^{3(1-f)+\omega_0f})+O((2n^f)^{2-2f})F(2(2n^f)^f) \right] \\
    & = O(n^{3(1-f)+\omega_0f}) + O(2^{3(1-f)+\omega_0f}n^{2+f-(3-\omega_0)f^2}) + O(2^{2-2f}n^{2-2f^2})F(2^{1+f}n^{f^2}).
    \endaligned
\end{equation}
Going an additional $k$ steps down, and letting $\alpha_f = 3(1-f)+\omega_0f$ and $S_k = \sum_{i=1}^k f^i$, we have 
\begin{equation}\label{eqn: sum_recursively_two}
    \aligned
    F(n) = O(n^{\alpha_f}) &+ \sum_{j=0}^k O(2^{2j + \alpha_f + (1-(3-\omega_0)f)S_j} n^{2+f^{j+1}-(3-\omega_0)f^{j+2}}) \\
     &+ O(2^{2(k+1)-2 S_{k+1}} n^{2-2f^{k+2}})F(2^{1+S_{k+1}}n^{f^{k+2}}).
    \endaligned
\end{equation}
Each term in this sum represents the number of flops required to handle block multiplications on problems of size $2^{1+S_j}n^{f^{j+1}}$.
Note that $k = 0$ recovers \eqref{eqn: sum_recursively_one}. \\
\indent To simplify this expression, we note that 
\begin{equation}\label{eqn: ratio_of_n_terms}
    2+f^{j+2} - (3-\omega_0)f^{j+3} = [2+f^{j+1} - (3-\omega_0)f^{j+2}] - [f^{j+1}(1-f)(1-(3-\omega_0)f)]
\end{equation}
where $f^{j+1}(1-f)(1-(3-\omega_0)f) > 0$ since $0 < f < 1$. In other words, the powers of $n$ in the sum of \eqref{eqn: sum_recursively_two} are decreasing in $j$. At the same time, treating $f$ as a constant so that $2^{S_j} = O(1)$, the corresponding powers of two satisfy 
\begin{equation}\label{eqn: bound_powers_of_two}
    2^{2j+\alpha_f+(1-(3-\omega_0)f)S_j}  = O(4^j).
\end{equation}
Together, these observations allow us to bound \eqref{eqn: sum_recursively_two} as
\begin{equation}\label{eqn: improve_bound}
    \aligned
    F(n) &= O(n^{\alpha_f}) + \sum_{j=0}^kO(4^jn^{2+f-(3-\omega_0)f^2}) + O(2^{2(k+1)-2 S_{k+1}} n^{2-2f^{k+2}})F(2^{1+S_{k+1}}n^{f^{k+2}}) \\
    & = O(n^{\alpha_f}) + O(4^kn^{2+f-(3-\omega_0)f^2}) + O(2^{2(k+1)-2S_{k+1}}n^{2-2f^{k+2}})F(2^{1+S_{k+1}}n^{f^{k+2}}).
    \endaligned 
\end{equation}
Assuming that Algorithm~\ref{alg:Jacobi3} defaults to direct diagonalization after $k+1$ recursive steps, equivalently that $2^{1+S_{k+1}}n^{f^{k+2}} \leq n_{\text{threshold}}$, we have $F(2^{1+S_{k+1}}n^{f^{k+2}}) \lesssim O(n_{\text{threshold}}^3) = O(1)$, which implies a final, non-recursive bound
\begin{equation}\label{eqn: improve_bound3}
    \aligned
    F(n) &= O(n^{\alpha_f}) + O(4^kn^{2+f-(3-\omega_0)f^2}) + O(2^{2(k+1)-2S_{k+1}}n^{2-2f^{k+2}}) \\
    &= O(n^{\alpha_f}) + O(4^kn^{2+f-(3-\omega_0)f^2}).
    \endaligned
\end{equation}
\indent One question remains:\ how large must $k$ be to guarantee $2^{1+S_{k+1}} n^{f^{k+2}} \leq n_{\text{threshold}}$? We have:
\begin{equation}\label{eqn: bound_k}
    \aligned
    \hspace{1mm} 2^{1+S_{k+1}}n^{f^{k+2}} \leq n_{\text{threshold}} &\iff \sum_{i=0}^{k+1}f^i + \log_2(n) f^{k+2} \leq \log_2(n_{\text{threshold}}) \\
    & \iff \frac{1-f^{k+2}}{1-f} + \log_2(n)f^{k+2} \leq \log_2(n_{\text{threshold}}) \\
    & \iff f^{k+2} \left[ (1-f)\log_2(n) - 1\right] \leq (1-f)\log_2(n_{\text{threshold}})-1.
    \endaligned
\end{equation}
Since our restriction on $f$ and $n_{\text{threshold}}$ ensures that both $(1-f)\log_2(n)-1$ and $(1-f)\log_2(n_{\text{threshold}})-1$ are positive, \eqref{eqn: bound_k} is equivalent to 
\begin{equation}\label{eqn: bound_k_2}
    f^{k+2} \leq \frac{(1-f)\log_2(n_{\text{threshold}}) - 1}{(1-f)\log_2(n) - 1}.
\end{equation}
Hence, it is sufficient to take
\begin{equation}\label{eqn: bound_k_3}
    k \geq \frac{\log_2((1-f)\log_2(n)-1) - \log_2((1-f)\log_2(n_{\text{threshold}})-1)}{\log_2(1/f)} - 2,
\end{equation}
meaning we can assume $k$ is at most $\frac{\log_2((1-f)\log_2(n))}{\log_2(1/f)}$ plus a constant. We can now bound $4^k$ asymptotically as
\begin{equation}\label{eqn: bound_4k}
    4^k \lesssim  \left(2^{\log_2((1-f)\log_2(n))}\right)^{2/\log_2(1/f)} = \left[(1-f)\log_2(n)\right]^{2/\log_2(1/f)}.
\end{equation}
In other words, $4^k$ is at most polylogarithmic in $n$. Since $n^{2+f-(3-\omega_0)f^2}$ is a (constant) power of $n$ smaller than $n^{\alpha_f}$, specifically
\begin{equation}\label{eqn: power_rations}
n^{2+f-(3-\omega_0)f^2} \cdot n^{-\alpha_f} = n^{-(1-f)(1-(3-\omega_0)f)}
\end{equation}
this suffices to show that $O(4^k n^{2+f-(3-\omega_0)f^2})$ is dominated by $O(n^{\alpha_f})$ as $n \rightarrow \infty$. \\
\indent Since both Strassen-like matrix multiplication and fast, recursive LUPP (as formulated in \cite{demmel2007fast}) attain communication lower bounds for their respective operations, extending this analysis to $W(n)$ is straightforward. In particular, the communication cost associated with a single $2b \times 2b$ matrix multiplication or a call to LUPP is $\Theta(\frac{b^{\omega_0}}{M^{\omega_0/2 - 1}}) = \Theta(\frac{n^{f\omega_0}}{M^{\omega_0/2 - 1}})$ (see \cite{BDHS14} for the details). This implies a communication analog of \eqref{eqn:Jacobi1}:
\begin{equation}\label{eqn: recursive_comm}
    W(n) = O \left( \frac{n^{3(1-f)+\omega_0f}}{M^{\omega_0/2-1}} \right) + O(n^{2-2f})W(2n^f).
\end{equation}
We can now simply repeat the argument given above, noting that the base case can be handled entirely in fast memory since $n_{\text{threshold}} \leq \sqrt{M/2}$. 
\end{proof} 

\vspace{3mm} 

\indent Theorem~\ref{Thm:complexity_recursive_Jacobi} implies that Algorithm~\ref{alg:Jacobi3} can reach near-optimal complexity in both arithmetic and communication. It should be noted that the constants suppressed in these bounds are dependent on both $f$ and the fast matrix multiplication routine used. \textcolor{black}{In particular, to apply \eqref{eqn: bound_powers_of_two} we absorbed the constant $2^{\alpha_f + (1-(3-\omega_0)f)S_j}$, where
\begin{equation}
    \alpha_f + (1-(3-\omega_0)f)S_j \leq 3(1-f)+\omega_0 f + (1-(3-\omega_0)f) \cdot \frac{f}{1-f} = \frac{3+(\omega_0-5)f}{1-f}.
\end{equation} 
In this sense, the algorithm is ``galactic;" that is, for fixed $\omega_0$ the hidden constant has exponential dependence on $(1-f)^{-1}$.}

\begin{remark}
    \normalfont{Note that if QRCP replaces LUPP in Algorithm~\ref{alg:Jacobi3} then lines~11-13 require $O((2b)^3) = O(n^{3f})$ flops and \eqref{eqn: total_non_recursive_flops} becomes $O(n^{3(1-f)+\omega_0f} + n^{2+f})$, which grows to $O(n^3)$ as $f \rightarrow 1$. In this way, QRCP prevents recursive Jacobi from reaching optimal complexity. Indeed, like the non-recursive blocked algorithm, $n^{3(1-f)+\omega_0 f} + n^{2+f}$ is always at least $n^{2.5}$.}
\end{remark} 
\indent The relegation to near-optimal complexity in Theorem~\ref{Thm:complexity_recursive_Jacobi} is a consequence of the condition $(1-f)^{-1} < \log_2(n_{\text{threshold}})$, which places a restriction on how far we can reduce $3(1-f)+\omega_0f$. When $n_{\text{threshold}}$ is relatively small, which recall depends on the available fast memory $M$, this may be fairly limiting. In practice, it can be relaxed by allowing $f$ to vary through the recursion -- for example by taking a sequence of log block sizes $f_1 \geq f_2 \geq \cdots $, where at recursive step $i$ (with $i=1$ corresponding to the highest level) the algorithm works with $f_i$ but passes $f_{i+1}$ in its recursive calls. Setting $f_i = f$ for all $i$ recovers Algorithm~\ref{alg:Jacobi3}.\\
\indent Repeating the analysis above, and reusing the notation $\alpha_f = 3(1-f)+\omega_0f$, we can easily compute the arithmetic complexity $F(n)$ for such a formulation of recursive Jacobi:
\begin{equation}\label{eqn: varying_f_recursive}
    \aligned
    F(n) &= O(n^{3(1-f_1)+\omega_0f_1})+O(n^{2-2f_1})F(2n^{f_1}) \\
    & = O(n^{\alpha_{f_1}}) + O(2^{3(1-f_2)+\omega_0f_2}n^{2+f_1-(3-\omega_0)f_1f_2})+ O(2^{2-2f_2}n^{2-2f_1f_2})F(2^{1+f_2}n^{f_1f_2}) \\
    & = O(n^{\alpha_{f_1}}) + \sum_{j=0}^k O(2^{2j+\alpha_{f_{j+2}}+(1-(3-\omega_0)f_{j+2})S_{j+1}'} \cdot n^{2+\prod_{i=1}^{j+1}f_i - (3-\omega_0) \prod_{i=1}^{j+2}f_i}) \\
    & \; \; \; \; \; \; \; \; \; \; \; \; \; \;  \; \; \; + O(2^{2(k+1)-2S_{k+2}'} \cdot n^{2 - 2\prod_{i=1}^{k+2}f_i})F(2^{1+S_{k+2}'} \cdot n^{\prod_{i=1}^{k+2} f_i}).
    \endaligned
\end{equation}
Here, $S'_i = \sum_{j=0}^{i-2} \prod_{l=1}^{i-j} f_l$ with $S_1' = 0$ (to be compared with $S_i$ from \eqref{eqn: sum_recursively_two}). Since we can again bound the powers of two in this sum by $O(4^j)$,\footnote{Note in particular that $S_i' \leq -1 + \sum_{j=0}^{i-1} f_1^j \leq -1+\frac{1}{1-f_1}$ for all $i >1$ since the $f_i$'s are decreasing.} the same argument used to prove Theorem~\ref{Thm:complexity_recursive_Jacobi} will go through here, and imply $F(n) = O(n^{3(1-f_1)+\omega_0f_1})$, provided the following hold:
\begin{enumerate}
    \item The powers of $n$ in the sum \eqref{eqn: varying_f_recursive} are decreasing in $j$.
    \item We reach $n_{\text{threshold}}$ after roughly $O(\log(\log(n))$ recursive steps.
\end{enumerate}
Item one is guaranteed provided
\begin{equation}\label{eqn: varyingf_conditon1}
    f_{i+1} > 1 - \frac{1-f_i}{(3-\omega_0)f_i},
\end{equation}
equivalently as long as the log block size does not shrink too aggressively. The latter, meanwhile, can be accomplished by fixing $f_i$ after a constant number of recursive steps. If, for example, $f_i = f$ for $i \geq p$ and $(1-f)^{-1}<\log_2(n_{\text{threshold}})$, the same argument made in the proof of Theorem~\ref{Thm:complexity_recursive_Jacobi} will imply $2^{1+S_{k+2}'} \cdot n^{\prod_{i=1}^{k+2}}f_i \leq n_{\text{threshold}}$ for $k-p \lesssim \log_2(\log_2(n))$ if $n$ is sufficiently large -- that is, large enough for the problem size to have decreased through the first $p-1$ recursive steps. \\
\indent This is of course the same restriction as in Theorem~\ref{Thm:complexity_recursive_Jacobi}, only now it applies a few steps into the recursion instead of immediately. While this will propagate to a lower bound on $1-f_1$ via \eqref{eqn: varyingf_conditon1}, we can simply increase $p$ to compensate (the tradeoff being an even larger constant hidden by the big-O). As in the proof of Theorem~\ref{Thm:complexity_recursive_Jacobi}, we can repeat this argument with \eqref{eqn: recursive_comm} to similarly remove the restriction on $f$ from the communication bound. \\
\indent Of course, $f$ is not the only aspect of recursive Jacobi that can be adjusted dynamically. Another option is the characterization of ``small enough" in line~19. That is, we may choose to be more/less strict with the requirement for convergence depending on the problem size and recursion level. In practice, different choices for the log block size and convergence criteria will impact performance. We explore this empirically in Section~\ref{section: numerical_ex} and Appendix~\ref{sec:Add_numerical_ex} but leave rigorous tuning (e.g., \cite{cho2025surrogate}) to future work.

\section{Jacobi SVD}\label{section: svd}

As mentioned in Section \ref{sec:intro}, Jacobi's method can be used to compute the SVD of an arbitrary matrix $\mathbf{G} \in {\mathbb R}^{m \times n}$. Indeed, Jacobi-based algorithms are included in the LAPACK SVD drivers as \texttt{xGEJSV} and \texttt{xGESVJ}, with details in 
\cite{DrmacVeselic08a,DrmacVeselic08b}.
These algorithms are popular for the same reasons as standard Jacobi:\ they can be implemented
efficiently and retain its accuracy advantages,
i.e., the ability to compute small singular values
to high relative accuracy
\cite{DemmelVeselic92,DGESVD99}. 
Accordingly, we dedicate this section to discussing Jacobi SVD algorithms and their associated arithmetic/communication complexities. For simplicity, we assume throughout that $m \geq n$.\\
\indent The key insight here is that the SVD of $\mathbf{G}$ can be obtained from a diagonalization of $\mathbf{G}^T\mathbf{G}$, recalling that the right eigenvectors of $\mathbf{G}^T\mathbf{G}$ are right \textit{singular} vectors of $\mathbf{G}$. That is, if $\mathbf{G}^T \mathbf{G} = \mathbf{V} \mathbf{D} \mathbf{V}^T$ for an orthogonal matrix $\mathbf{V}$ and diagonal matrix $\mathbf{D}$, then $\mathbf{G}\mathbf{V} = \mathbf{U} {\bm \Sigma}$ for the SVD $\mathbf{G} = \mathbf{U} {\bm \Sigma} \mathbf{V}^T$. Intuitively, Jacobi-based SVD algorithms obtain $\mathbf{V}$ by applying Jacobi's method (classical, blocked, or recursive) to $\mathbf{G}^T\mathbf{G}$. In general, this is done without forming the Gram matrix explicitly. Instead, submatrices of $\mathbf{G}^T \mathbf{G}$, which are again $2b \times 2b$ and may be handled in a cyclic/random/dynamic order, are computed ``on-the-fly'' from (block) columns of $\mathbf{G}$. The input matrix $\mathbf{G}$ is then transformed to $\mathbf{U}{\bm \Sigma}$ by applying the corresponding rotation matrices on the right \textit{only}. Accordingly, this method for computing the SVD is often referred to as ``one-sided Jacobi," which we present here as Algorithm~\ref{alg:Jacobi_SVD}. Note that this routine can be built on top of any version of Jacobi from the preceding sections. 

\begin{algorithm}[t]
\caption{One-Sided Jacobi SVD}
\label{alg:Jacobi_SVD} \begin{algorithmic}[1]
\Require $\mathbf{G} \in {\mathbb R}^{m \times n}$ with $m \geq n$ and columns partitioned as follows (for $b$ that divides $n$):
\[ \mathbf{G}(:,I) = \mathbf{G}(:, (I-1)b+1:Ib) \; \; \; \text{for} \; \; \; 1 \leq I \leq n/b. \]
 
\Ensure On output, $\mathbf{G} = \mathbf{U} {\bm \Sigma}\mathbf{V}^T$ is an approximate (reduced) singular value decomposition of $\mathbf{G}$, with $\mathbf{U} \in {\mathbb R}^{m \times n}$ and ${\bm \Sigma},\mathbf{V} \in {\mathbb R}^{n \times n}$. ${\bm \Sigma}$ is diagonal and contains approximations of the singular values of $\mathbf{G}$.
\algrule
\setstretch{1.1}
\Function{$[\mathbf{U},{\bm \Sigma},\mathbf{V}]=$
Jacobi\_SVD}{$\mathbf{G}$} \State $\mathbf{V}=\mathbf{I}_n$
\Repeat \For{all $1 \leq I < J \leq n/b $ in some order,} \State $\mathbf{\hat{A}}= \mathbf{G}(:, [I,J])^T \mathbf{G}(:,[I,J])$ \Comment{$\mathbf{\hat{A}}$ is a $2b \times 2b$ submatrix of $\mathbf{G}^T\mathbf{G}$ formed on the fly} \If{$\mathbf{\hat{A}}$ is far
enough from diagonal} \State $\mathbf{\hat{A}}=\mathbf{\hat{V}}\mathbf{\hat{D}}\mathbf{\hat{V}}^{T}$
is an eigendecomposition of $\mathbf{\hat{A}}$ \Comment{ e.g., from Algorithms~\ref{alg:Jacobi1},\ref{alg:Jacobi2}, or \ref{alg:Jacobi3} } 
\State $\mathbf{\hat{V}}_1 = \mathbf{\hat{V}}(1:b,1:2b)$
\State $[\mathbf{P},\sim,\sim] =$ \textsc{Recursive}\_LUPP$(\mathbf{\hat{V}}_1^T)$
\rlap{\smash{$\left.\begin{array}{@{}c@{}}\\{}\\{}\end{array}\color{black}\right\}%
          \color{black}\begin{tabular}{l}Optional\end{tabular}$}}
\State $\mathbf{\hat{V}} = \mathbf{\hat{V}}\mathbf{P}^T$
\State Multiply block columns $I$ and $J$ of $\mathbf{G}$ by $\mathbf{\hat{V}}$
\State Multiply block columns $I$ and
$J$ of $\mathbf{V}$ by $\mathbf{\hat{V}}$ \EndIf \EndFor \Until{the columns of $\mathbf{G}$ are nearly orthogonal} 
\State ${\bm \Sigma} = \text{diag}(||\mathbf{G}(:,1)||_2, \ldots, ||\mathbf{G}(:,n)||_2)$ 
\State $\mathbf{U} = \mathbf{G}{\bm \Sigma}^{\dagger}$ \EndFunction \end{algorithmic} 
\end{algorithm} 

\indent Running Jacobi implicitly on $\mathbf{G}^T\mathbf{G}$ is key to the numerical stability of Algorithm~\ref{alg:Jacobi_SVD}. In particular, forming the Gram matrix $\mathbf{G}^T\mathbf{G}$ squares the condition number (since $\kappa_2(\mathbf{G}^{T}\mathbf{G})=\kappa_2(\mathbf{G})^{2}$),
which can significantly impact the accuracy of the smaller singular values of $\mathbf{G}$ \cite{DrmacVeselic08a,DGESVD99}. Since the popularity of Jacobi-based SVD algorithms is rooted in their ability to compute these singular values to higher precision, avoiding this potential instability is critical. Practical considerations are also a concern, as storing and manipulating an $n \times n$ Gram matrix may present a computational/memory bottleneck in large-scale calculations~\cite{ActaNumerica2014}, though this of course may arise anyways if $b$ is a significant fraction of $n$.\\
\indent The name ``one-sided-Jacobi'' further distinguishes Algorithm~\ref{alg:Jacobi_SVD} from the \textcolor{black}{Kogbetliantz SVD algorithm \cite{Kogbetliantz-1955}}, in which Jacobi is modified to (1) take a nonsymmetric (and potentially rectangular) input matrix and (2) compute an SVD of each subproblem instead of a diagonalization -- e.g., in line 7 of Algorithm~\ref{alg:Jacobi1}/Algorithm~\ref{alg:Jacobi2} or the base case of Algorithm~\ref{alg:Jacobi3}. We focus on the one-sided version for two reasons. First, it is the more commonly used Jacobi-based SVD algorithm. Second, since the \textcolor{black}{Kogbetliantz approach} does not significantly change the structure of Jacobi -- save doubling its memory footprint to allow for \textcolor{black}{different} left and right transformations -- the complexity analysis from the previous sections carries over directly.  \\
\indent Convergence of one-sided Jacobi SVD follows from convergence of classical/blocked Jacobi, provided a cyclic ordering is used on the blocks of $\mathbf{G}^T\mathbf{G}$ and the optional pivoting step in Algorithm~\ref{alg:Jacobi_SVD} is executed. In practice, convergence is accelerated by first computing a QR factorization of $\mathbf{G}$ and running one-sided Jacobi SVD on the Gram matrix of the corresponding R-factor (or its transpose) as follows:
\begin{enumerate}
    \item Pre-process: $\mathbf{G}\mathbf{P} = \mathbf{Q} \begin{pmatrix} \mathbf{R} \\ 0 \end{pmatrix}$
    \item Call Jacobi SVD: $[\mathbf{\hat{V}},{\bm \Sigma},\mathbf{\hat{U}}] = \text{Jacobi\_SVD}(\mathbf{R}^T)$
    \item Assemble: $\mathbf{U} = \mathbf{Q}\begin{pmatrix} \mathbf{\hat{U}} \\ 0 \end{pmatrix}$ and $\mathbf{V} = \mathbf{P}\mathbf{\hat{V}}$ 
\end{enumerate}
The goal here is to both cut down the size of the problem, particularly when $m\gg n$, and run Jacobi on a matrix that is closer to diagonal than $\mathbf{G}^T\mathbf{G}$. To guarantee the latter, the QR factorization must be computed with column pivoting, which, as discussed at length in \textcolor{black}{Remark~\ref{rem: QRCP}}, can present a computational/communication bottleneck. \textcolor{black}{Of course, preprocessing even without pivoting is a good idea, since $\mathbf{R}$ can still be much smaller than $\mathbf{G}$}. Further justification for the details of this approach, particularly the choice to run Jacobi SVD on $\mathbf{R}^T$ instead of $\mathbf{R}$, are discussed in \cite{DrmacVeselic08a}. Note also that this heuristic can be applied within Algorithm~\ref{alg:Jacobi_SVD} itself, where we might compute a QR factorization of $\mathbf{G}(:,[I,J])$ and form $\mathbf{\hat{A}}$ from its R-factor, which can be done stably with optimal communication \cite{DemmelGrigoriHoemmen2012}. \\
\indent We are now ready to state arithmetic/communication complexity bounds for Algorithm~\ref{alg:Jacobi_SVD}. We start with a fairly straightforward flop count. For brevity we omit the proof (see \cite{arxiv_manuscript}). The main addition here over standard Jacobi is the cost of forming each $2b \times 2b$ submatrix of $\mathbf{G}^T \mathbf{G}$, which -- like the subsequent block multiplications in lines 11 and 12 -- can be done via $\lceil \frac{m}{2b} \rceil$ multiplications of $2b \times 2b$ matrices.

\begin{prop}\label{prop: arithmetic_Jacobi_SVD}
    Let $F(m,n)$ denote the arithmetic complexity of \textcolor{black}{one sweep of}  Algorithm~\ref{alg:Jacobi_SVD} under the following conditions:
    \begin{enumerate}
        \item $O(n^{\omega_0})$ matrix multiplication is used.
        \item The method used to diagonalize $\mathbf{\hat{A}}$ in line~7 requires \textcolor{black}{$O(G(2b))$} flops.
    \end{enumerate}
    Then $F(m,n) = O\left(mn^2b^{\omega_0-3} + \lceil \frac{n}{b} \rceil^2 G(2b)\right)$.
\end{prop}
\begin{proof}
    For each pair $I < J$ in a sweep of Algorithm~\ref{alg:Jacobi_SVD}, we note the following steps and their associated arithmetic costs: 
    \begin{itemize}
        \item Forming $\M{\hat{A}}$ (line 5) -- $\Theta(mb^{\omega_0-1})$ flops. 
        \item The subsequent diagonalization (line 7) -- $G(2b)$ flops.
        \item the optional pivoting step (lines 8-10) -- $O(b(2b)^{\omega_0-1})$ flops
        \item The final block multiplications(lines 11-12) -- $\Theta(mb^{\omega_0-1})$ flops.
    \end{itemize}
    Since there are $\Theta(\lceil \frac{n}{b} \rceil^2)$ pairs $(I,J)$ in total, the complexity of one sweep of Algorithm~\ref{alg:Jacobi_SVD} is
    \begin{equation}\label{eqn: sweep_count_SVD}
        O\left(\left\lceil \frac{n}{b} \right\rceil^2 \left(mb^{\omega_0-1} + G(2b) + (2b)^{\omega_0}  \right) \right) = O\left( mn^2b^{\omega_0-3} + \left\lceil \frac{n}{b}\right\rceil^2 G(2b) \right) .
    \end{equation}
    Noting that lines 16 and 17 require only $\Theta(mn)$ flops,  we conclude that $F(m,n)$ is equal to \eqref{eqn: sweep_count_SVD}.
    \end{proof} \\

\textcolor{black}{This result is stated in terms of a black-box function $G(\cdot)$, which bounds the arithmetic complexity of the subroutine used to compute the diagonalization in line 7. Accordingly, it is immediately compatible with the complexities derived in the preceding sections --  e.g., we can take $G(n) = n^3$ if Algorithm~\ref{alg:Jacobi1} is used. In fact, Proposition~\ref{prop: arithmetic_Jacobi_SVD} implies that versions of one-sided Jacobi built on top of these algorithms can be formulated to attain similar arithmetic complexities.} In the recursive setting, this requires taking $b = n^f$, in which case $mn^2b^{\omega_0-3} = mn^{3(1-f)+\omega_0f-1}$, mirroring the leading term from Theorem~\ref{Thm:complexity_recursive_Jacobi} (indeed equal to it when $m = n$). \\
\indent We consider next communication complexity. Again, the main addition here is the cost associated with forming each $2b \times 2b$ submatrix of $\mathbf{G}^T\mathbf{G}$. 
Propositions~\ref{prop: communication_Jacobi_SVD_2} -~\ref{prop: communication_Jacobi_SVD_3} extend our communication bounds from the previous sections, in particular for increasing block size $b$. We again defer the proofs, which are straightforward extensions of the communication bounds from the preceding sections, to  \cite{arxiv_manuscript}. \textcolor{black}{As in Proposition~\ref{prop: arithmetic_Jacobi_SVD}, when the subproblem $\mathbf{\hat{A}}$ is too large to first in fast memory (i.e., Proposition~\ref{prop: communication_Jacobi_SVD_3}) our result relies on a black-box function $U(\cdot)$ that bounds the communication complexity of the algorithm performing the diagonalization.}

\begin{prop}\label{prop: communication_Jacobi_SVD_2}
    Let $W(m,n)$ denote the serial communication complexity of \textcolor{black}{one sweep of} Algorithm~\ref{alg:Jacobi_SVD} with $b = 1$. 
    If $M$ is the size of available fast memory and $\sqrt{M}\leq m < M$, 
    then $W(m,n) = \Omega(m^2n^2/M)$. This lower bound is attainable.
\end{prop}
\begin{proof}
    This follows from the same argument used to prove Theorem~\ref{thm: classical_lower_bound}. In this case, reading in $\Theta(M/m)$ columns of $\M{G}$ allows us to work through an $\Theta(M/m) \times \Theta(M/m)$ block of $\M{G}^T \M{G}$. We can do at most $\Theta(M^2/m)$ flops on this data -- i.e., $\Theta(m)$ flops for each of the $\Theta(M^2/m^2)$ currently available off-diagonal entries of $\M{G}^T\M{G}$. Since Proposition~\ref{prop: arithmetic_Jacobi_SVD} implies that a total of $O(mn^2)$ flops are required by Algorithm~\ref{alg:Jacobi_SVD} with $b = 1$, we conclude that $\Omega(mn^2/(M^2/m)) = \Omega(m^2n^2/M^2)$ references are required. Since each of these references corresponds to $M$ reads and writes, we conclude $W(m,n) = \Omega(m^2n^2/M)$.
\end{proof}

\begin{prop}\label{prop: communication_Jacobi_SVD}
    Let $W(m,n)$ denote the serial communication complexity of \textcolor{black}{one sweep of} Algorithm~\ref{alg:Jacobi_SVD} under the following conditions:
    \begin{enumerate}
        \item $O(n^{\omega_0})$ matrix multiplication is used.
        \item $b = O(\sqrt{M})$ is chosen so that lines 7-10 can take place entirely in fast memory of size $M$.
    \end{enumerate}
    Then $W(m,n) = O(mn^2/b)$, and in particular $W(m,n) = O(mn^2/\sqrt{M})$ for $b = \Theta(\sqrt{M})$.
\end{prop}
\begin{proof}
    For each $I < J$, forming $\M{\hat{A}}$ in line 5 requires $\Theta(mb^{\omega_0-1})$ flops. If we do this by allocating $\M{\hat{A}}$ and updating it by (1) reading in each $2b \times 2b$ block of $\M{G}(:,[I,J])$ and (2) adding to $\M{\hat{A}}$ the product of the block and its transpose, we can execute line 5 with computational intensity $(2b)^{\omega_0}/(2b)^2 = (2b)^{\omega_0-2}$. Hence, the communication required to obtain $\M{\hat{A}}$ is $O(mb^{\omega_0-1}/b^{\omega_0-2}) = O(mb)$. In each sweep of Algorithm~\ref{alg:Jacobi_SVD}, the cumulative communication attributable to line 5 is therefore $O(\lceil \frac{n}{b} \rceil^2 mb) = O(mn^2/b)$. Since the communication cost of lines 11 and 12 is the same as line 5, and moreover since the diagonalization/optional pivoting in lines 7-10 can be done entirely in fast memory, we conclude $W(m,n) = O(mn^2/b)$
\end{proof} 

\begin{prop}\label{prop: communication_Jacobi_SVD_3}
Let $W(m,n)$ denote the serial communication complexity of \textcolor{black}{one sweep of} Algorithm~\ref{alg:Jacobi_SVD} under the following conditions:
\begin{enumerate}
    \item $O(n^{\omega_0})$ matrix multiplication is used.
    \item $b > \sqrt{M}/2$ for $M$ the size of available fast memory.
    \item The method used to diagonalize $\mathbf{\hat{A}}$ in line~7 has communication cost $O(U(2b))$.
\end{enumerate}
Then  $W(m,n) = O\left(\frac{mn^2b^{\omega_0-3}}{M^{\omega_0/2-1}} + \lceil \frac{n}{b} \rceil^2U(2b)\right)$.    
\end{prop}
\begin{proof}
    As in the proof of Theorem~\ref{Thm:complexity_recursive_Jacobi} we apply the communication bounds from \cite{BDHS12-FastLA}, noting that, for $b > \sqrt{M}/2$, neither $2b \times 2b$ matrix multiplication nor the diagonalization in line 7 can take place in fast memory. For each pair $I < J$ in the main loop of Algorithm~\ref{alg:Jacobi_SVD} we have the following communication costs:
    \begin{itemize}
        \item Line~5 (done as $\lceil \frac{m}{2b} \rceil$ square $2b \times 2b$ multiplications): $O( \frac{mb^{\omega_0-1}}{M^{\omega_0/2 - 1}})$
        \item Line~7: $U(2b)$
        \item Lines~9-10: $O(\frac{b^{\omega_0}}{M^{\omega_0/2 - 1}})$
        \item Lines~11-12 (again done in $2b \times 2b$ pieces): $O(\frac{mb^{\omega_0-1}}{M^{\omega_0/2 -1}})$.
    \end{itemize}
    Hence, each sweep of one-sided Jacobi has communication complexity
    \begin{equation}\label{eqn: one_sweep_comm}
        O\left(\left\lceil \frac{n}{b}\right\rceil^2 \left( \frac{mb^{\omega_0-1}}{M^{\omega_0/2 -1}} + U(2b)\right) \right) = O\left( \frac{mn^2b^{\omega_0-3}}{M^{\omega_0/2 - 1}} + \left\lceil \frac{n}{b} \right\rceil^2 U(2b)\right).
    \end{equation}
    Since the communication associated with lines 16 and 17 is negligible compared to \eqref{eqn: one_sweep_comm}, we conclude that this is a bound for $W(m,n)$.
\end{proof} \\

\indent Note that the bound in Proposition~\ref{prop: communication_Jacobi_SVD_2} is smaller than that of Proposition~\ref{prop: communication_Jacobi_SVD} only when $m < \sqrt{M}$, in which case $\mathbf{G}$ fits in fast memory (recall that $m \geq n$) and the only communication required is the $mn$ words needed to read in the matrix. Like standard Jacobi, Proposition~\ref{prop: communication_Jacobi_SVD_3} implies that Algorithm~\ref{alg:Jacobi_SVD} can reach (nearly) optimal complexity by allowing the block size $b$ to grow with $n$, though this requires that the diagonalization in line~7 be done optimally itself (e.g., via recursive Jacobi).

\section{Numerical Examples}\label{section: numerical_ex}
We now test the performance of scalar, blocked, and recursive Jacobi with a handful of numerical experiments. The focus in this section is on tests that verify phenomena underpinning our theoretical results, specifically (1) that pivoting methods like QRCP and LUPP \textit{are} necessary in the blocked and recursive settings to guarantee convergence and (2) that when convergence is achieved, it occurs after what can reasonably be described as $O(1)$ sweeps. In Appendix~\ref{sec:Add_numerical_ex}, we present additional examples aimed at testing the methods with different recursive/blocking parameters. 
All experiments were conducted on a MATLAB implementation of Algorithms~\ref{alg:Jacobi1}--\ref{alg:Jacobi3}  with results obtained in MATLAB version R2024a. \\
\indent We start by discussing a few practical details of the implementation: 

\begin{itemize}
    \item \textbf{Convergence Metric:} In the pseudocode of Sections~\ref{section: classical}--\ref{section: recursive}, the convergence criteria for Jacobi were left intentionally vague. While it may seem natural to use the off-diagonal Frobenius norm \eqref{eqn: Omega_A} to monitor convergence,\footnote{Indeed, it is guaranteed to decrease monotonically as discussed in \cite{Jacobi1846,Drmac2009} and Section~\ref{section: classical} and was used to characterized convergence in our theoretical results.} this can introduce additional complexity in the blocked/recursive settings. In particular, it necessitates that each subproblem has a separate stopping rule, dependent not only on the desired final accuracy but also on the size of the subproblem. If this is not done, Jacobi may loop infinitely -- e.g., if the Frobenius norm of each subproblem lies below the convergence tolerance but their sum exceeds it. To avoid this complication, we adopt the largest off-diagonal entry, in absolute value, as our measure of convergence. Although this metric is \textit{not} guaranteed to decrease with each rotation, it is independent of problem size and therefore easier to implement. Moreover, a bound on the largest off-diagonal entry of $\M{A}$ immediately implies an upper bound on $\Xi(\M{A})$.\footnote{See the plots of Figure~\ref{fig:exp_comparison_bottom} in Appendix~\ref{sec:Add_numerical_ex} for an empirical comparison.} In each experiment, we stop (scalar, block, or recursive) Jacobi once 
    \begin{equation}\label{eqn: stopping_condition}
        \max_{i \neq j} \left| \M{A}^{(k)}_{ij} \right| \leq 10^{-7} \cdot \max_{i,j} \left| \M{A}^{(0)}_{ij} \right|,
    \end{equation}
    where $\M{A}^{(k)}$ denotes the matrix after the $k$-th sweep with $k = 0$ corresponding to the input matrix. Accordingly, convergence is  only checked at the end of each sweep. In the recursive case, to maintain consistency with the other methods, this condition is applied throughout the recursion with the right hand side fixed -- i.e., dependent on the \textit{full} input matrix. 

    \item \textbf{(Block) Ordering:} Unless otherwise stated, each algorithm works through off-diagonal entries/blocks according to a standard row-cyclic ordering (analagous to \eqref{eqn: column_cyclic}).
    
    \item \textbf{When to Rotate:} Consistent with the aforementioned convergence criteria, we apply a rotation to a subproblem $\M{\hat{A}}$ only if its maximal off-diagonal entry has magnitude above $10^{-7} \cdot \max_{i,j}|\M{A}_{ij}^{(0)}|$. This defines the criteria in lines 5, 6, and 9, respectively, of Algorithms~\ref{alg:Jacobi1}--\ref{alg:Jacobi3} (and applies uniformly regardless of the size of $\M{\hat{A}}$). 
    
    \item \textbf{Flop Count:} In addition to convergence data, we report an estimated flop count for each method tested. These are based on corresponding flop estimates for the key operations/kernels used in our Jacobi algorithms, as listed 
    in Table~\ref{tab:flops_approx}. Note that we use classical BLAS-3 estimates here rather than Strassen-like $O(n^{\omega_0})$ bounds for simplicity. While recursive algorithms and fast matrix multiplication may offer asymptotic advantages, these benefits typically manifest only for ``galactic" matrices. For the same reason, we use standard LUPP in place of its fast, recursive variant throughout. 
\end{itemize}

\renewcommand{\arraystretch}{1.2}
\begin{table}[t]
    \centering
    \begin{tabular}{lll}
        \toprule
        \textbf{Operation} & \textbf{Flops} & \textbf{Source}\\
        \midrule
        $\M{A}\M{B}$ for $\M{A} \in \mathbb{R}^{m \times n}$ and $\M{B} \in \mathbb{R}^{n \times p}$         & $mp(2n-1)$ & \cite[Chapter 2.6.1]{demmelAppliedLinearAlgebra}\\
        {Full $n \times n$ eigendecomposition}
 & $8\frac{2}{3}n^3$ & \cite[Chapter 5.3]{demmelAppliedLinearAlgebra}\\
        QRCP on $\M{A}, m\geq n$ &  $2mn^2 - \frac{2}{3} n^3$ & \cite[Chapter 3.2.2]{demmelAppliedLinearAlgebra}\\
        LUPP on $\M{A}, m\geq n$ &  $mn^2 - \frac{1}{3} n^3$ & \cite[Chapter 2.3]{demmelAppliedLinearAlgebra}\\
        \bottomrule
    \end{tabular}
    \caption{Approximate flop counts for the building blocks of Jacobi's method.}
    \label{tab:flops_approx}
\end{table}
\renewcommand{\arraystretch}{1}

\indent For each experiment, we take as input a random symmetric matrix, generated by drawing a standard Gaussian matrix $\M{G} \in {\mathbb R}^{n \times n}$ and setting $\M{A} = (\M{G}+\M{G}^T)/2$. The problem size $n$ is chosen to be relatively small so that each experiment can run quickly on a personal laptop (and, again, since the theoretical efficiency gains for blocked/recursive Jacobi will not kick in until the input matrix is large). In accordance with the conditions of Theorem~\ref{Thm:complexity_recursive_Jacobi}, we also take comparatively small log block sizes $f$ in Algorithm~\ref{alg:Jacobi3} -- i.e., not pushing $f \rightarrow 1$.

\begin{figure}[t]
    \centering
    \makebox[\textwidth][c]{%
        \includegraphics[width=\textwidth]{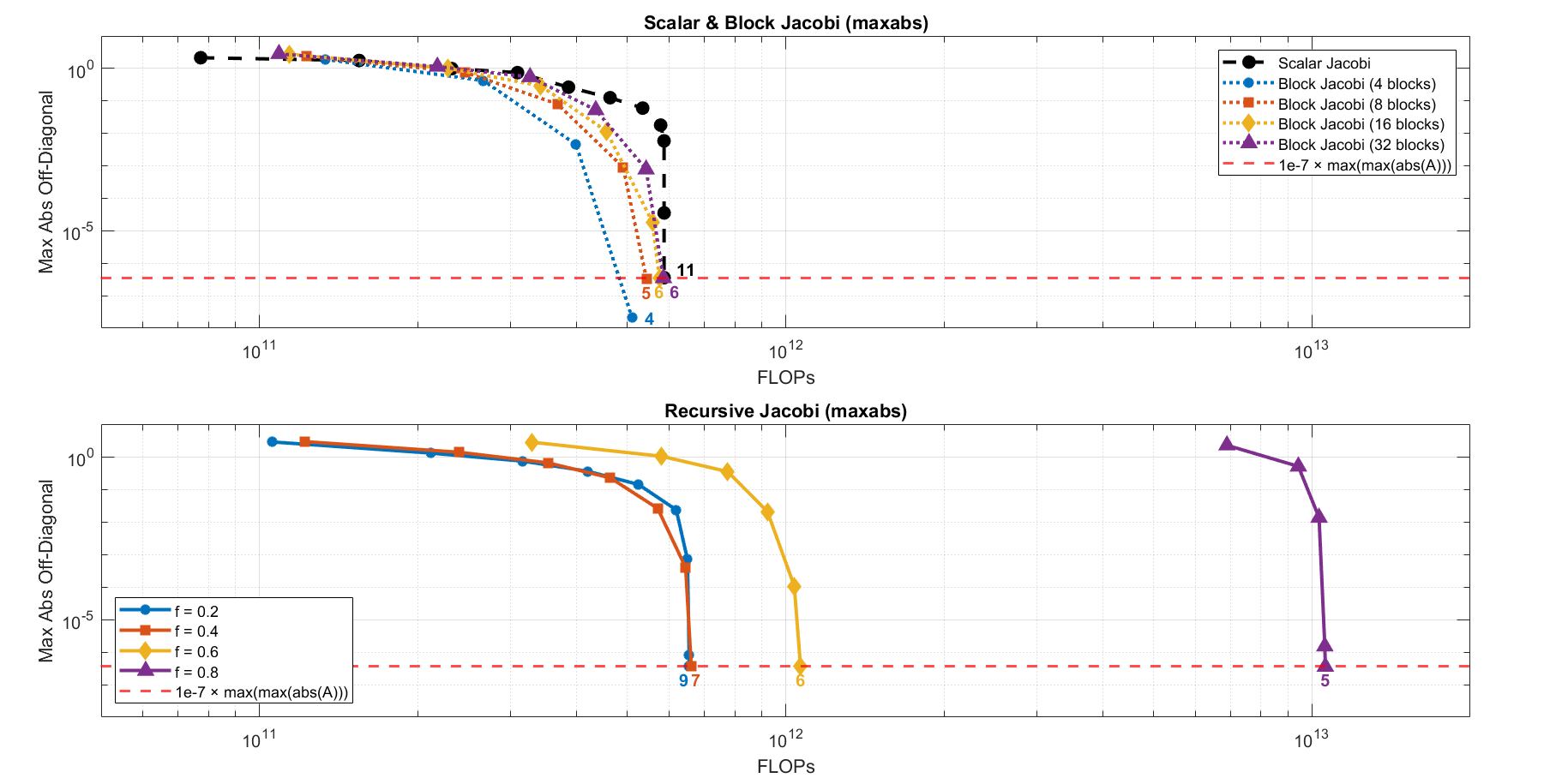}
    }
    \caption{Comparison of scalar, blocked and recursive Jacobi on a $512 \times 512$ input matrix. The final sweep count is listed at the end of each curve.}
    \label{fig:exp_comparison}
\end{figure}

\indent Our first experiment (Figure~\ref{fig:exp_comparison}) compares the scalar, blocked, and recursive versions of Jacobi on the same $512 \times 512$ input matrix. We include several versions of both Algorithm~\ref{alg:Jacobi2} (each with a different block size) and Algorithm~\ref{alg:Jacobi3} (each with a different \textit{log} block size). Both the blocked and recursive algorithms are done without applying either QRCP or LUPP to rotation matrices; as anticipated, since failure is unlikely to occur on a random input matrix, this does not prevent them from converging. For each method, we plot the current error (as measured by the largest off-diagonal entry) against an approximate flop count, with each data point corresponding to the end of one sweep. For easy comparison, we label the last of these data points by a total sweep count. Note that the recursive algorithms are all done with $n_{\text{threshold}} = 4$.  \\
\indent As expected, all nine algorithms included in Figure~\ref{fig:exp_comparison} converge in only a handful of sweeps. With the exception of the recursive algorithm with $f = 0.8$, which cannot reap the benefits of a large log block size given the choice of $n$, reaching this convergence requires around $10^{12}$ flops, which in this case is roughly $116n^3$. These observations confirm that recursive Jacobi can be an effective alternative to the classical and blocked algorithms even for moderate-scale problems, provided the thresholding and block size parameters are chosen appropriately. 
We explore these tuning parameters more extensively in Appendix~\ref{sec:Add_numerical_ex}. \\
\indent Our next experiment aims to demonstrate the efficacy of including a pivoting step like QRCP or LUPP in block/recursive Jacobi. For simplicity, we focus on the former, recalling that convergence of the blocked algorithm implies convergence for recursive Jacobi (see Proposition~\ref{prop:recursive_convergence}). Of course, as observed in the previous experiment, the blocked algorithm converges with no issue on random inputs, even when pivoting is omitted. Moreover, explicit failure examples like those discussed in Section~\ref{section: classical} are only known in the scalar setting. Accordingly, we verify the benefits of QRCP/LUPP by applying them to an ``adversarial" version of Algorithm~\ref{alg:Jacobi2}, specifically one that handles line 7 via an adversarial implementation of scalar Jacobi, which chooses rotation angles to explicitly violate the convergence condition of Forsythe and Henrici \cite{Forsythe_Henrici_1960}. This is done by adding $\pi/2$ to the angle $\theta$ coming from \eqref{eqn: roation_angle}. By applying a sequence of such ``sabotaged" $2\times 2$ rotations on each subproblem, we aim to drive Algorithm~\ref{alg:Jacobi2} towards failure. Since our adversarial scalar Jacobi algorithm itself is guaranteed to fail in certain cases, we cap the number of sweeps it performs on each subproblem.

\begin{figure}[t]
    \centering
    \makebox[\textwidth][c]{%
        \includegraphics[width=\textwidth]{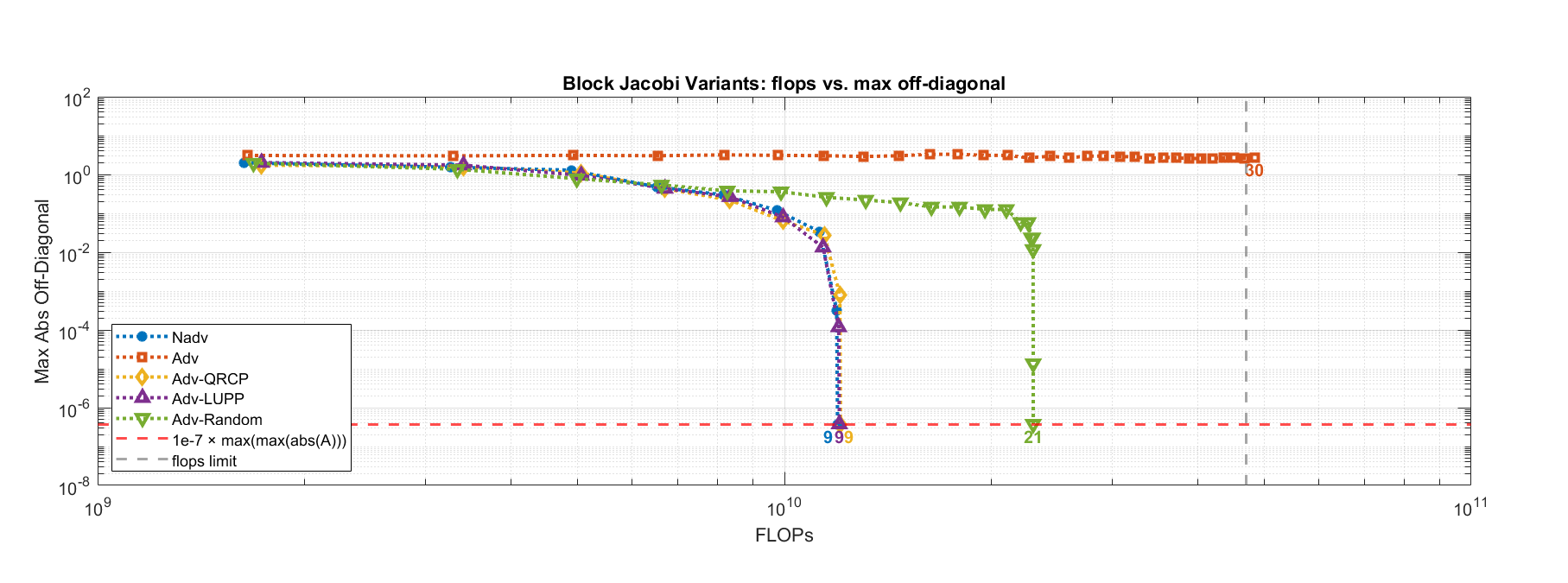}
    }
    \caption{Performance of various implementations of blocked Jacobi,  including Algorithm~\ref{alg:Jacobi2} without pivoting (Block Nadv) as well as an adversarial version both with pivoting (Block QRCP or Block LUPP) and without (Block Adv or Block Random, the latter handling off-diagonal blocks in random order). In each case, $n = 512$ and $b = 2$.}
    \label{fig:exp_comparisonADV}
\end{figure}

\indent Figure~\ref{fig:exp_comparisonADV} reports the performance of this adversarial blocked algorithm both with pivoting (i.e., ``Adv-QRCP" or ``Adv-LUPP") and without (``Adv"). The prefix ``Adv" or ``Nadv" indicates whether the ``sabotaged" rotation is applied when using Algorithm~\ref{alg:Jacobi1} within each block. For comparison, we also include ``Nadv", which refers to Algorithm~\ref{alg:Jacobi2} in its original form without pivoting and using Algorithm~\ref{alg:Jacobi1} for diagonalization, as well as ``Adv-Random", a variant of the adversarial blocked algorithm that processes off-diagonal blocks in random order, again without pivoting. The format of this plot is the same as in Figure~\ref{fig:exp_comparison}, with the input matrix again $512 \times 512$ and random.\\
\indent The results of this comparison are striking. Despite the fact that the input matrix is random, the adversarial implementation fails to converge (or really make any progress at all). In contrast, adding a pivoting step -- QRCP or LUPP -- is enough to guarantee  convergence, in particular without significantly driving up the flop count compared to Block Nadv. Choosing off-diagonal blocks randomly also yields convergence (echoing Proposition~\ref{prop: rand_converges}) though it requires additional sweeps and is therefore more expensive.  These results imply that both  pivoting and randomization effectively safeguard against failure, even when Jacobi is implemented adversarially. \\
\indent As mentioned above, we include additional numerical experiments in Appendix~\ref{sec:Add_numerical_ex}. As a quick summary, Table~\ref{tab:block-jacobi-sweeps} records the number of sweeps required by block Jacobi for different combinations of problem size $n$ and block size $b$, while Figures~\ref{fig:exp_comparison_bottom} and ~\ref{fig:exp_comparison_maxdepth} explore the convergence properties of recursive Jacobi with varying recursion depth/base-case diagonalization precisions.

\section{Conclusions and Open Problems}
This paper presented a detailed analysis of the arithmetic and communication complexity of Jacobi’s method for the symmetric eigenvalue problem and singular value decomposition. We examined classical, blocked, and recursive formulations of the algorithm with the aim of pushing its complexity towards that of matrix multiplication,
both in terms of computational efficiency and data movement. \\
\indent In a follow up to this work, we will extend our complexity analysis to the parallel setting, including with variable memory per processor. 
Other topics of interest for ongoing/future work include generalizations of Jacobi's method \cite{detherage2025unified,kolda2009tensor} as well as investigations of its compatibility with mixed-precision arithmetic \cite{Gao2025mixed, higham2025computing}. 

\section*{Acknowledgements}
This work was supported by NSF grant MSPRF 2402027 and partly supported by NSF grant DMS 2412403. Special thanks to Zlatko Drmač, Rikhav Shah, and Isabel Detherage for helpful discussions. \textcolor{black}{We also acknowledge two anonymous referees, who provided detailed feedback on an earlier draft of this paper.} For reproducibility, our code is available at \url{http://github.com/hrluo/RecursiveJacobi}.
\bibliographystyle{abbrv}
\bibliography{book-ref}

\appendix
\section{\textcolor{black}{Complexity Lower Bounds for Eigenvalue Problems}}\label{section: lower_bounds}
\textcolor{black}{In this appendix, we justify the claim that the symmetric eigenvalue problem -- and in fact the nonsymmetric eigenvalue problem as well -- is at least as difficult as matrix multiplication. We first note work of Bügisser, Karpinski, and Lickteig \cite{Burgisser1991}, which establishes that, given a symmetric matrix $\textbf{A} \in {\mathbb R}^{n \times n}$, the problem of finding an invertible $\textbf{S} \in {\mathbb R}^{n \times n}$ such that $\textbf{S} \textbf{A} \textbf{S}^T$ is diagonal (what they call the Orthogonal Basis Problem $OGB_n$) has complexity at least that of matrix multiplication. Since $\textbf{S}$ is not required to be orthogonal, $OGB_n$ is strictly easier than the symmetric eigenvalue problem.  \\
\indent More generally, we offer the following reduction of matrix multiplication to the \textit{nonsymmetric} eigenvalue problem. Given $\textbf{A}, \textbf{B} \in {\mathbb R}^{n \times n}$, consider the $3n \times 3n$ block matrix
\begin{equation}
    \textbf{M} = \begin{pmatrix} \textbf{D}_1 & \textbf{A} & 0 \\ 0 & \textbf{I} & \textbf{B} \\ 0 & 0 & \textbf{D}_3 \end{pmatrix},
\end{equation}
where $\textbf{D}_1$ and $\textbf{D}_3$ are diagonal with eigenvalues that are distinct from one another and distinct from 1. If $\lambda$ is the $i$-th entry on the diagonal of $\textbf{D}_3$, it is easy to see that $(\lambda, v)$ is an eigenpair of $\textbf{M}$ for
\begin{equation}
    v = \begin{pmatrix} (\lambda-1)^{-1}(\lambda \textbf{I} - \textbf{D}_1)^{-1} \textbf{A} \textbf{B} e_i \\ (\lambda-1)^{-1}\textbf{B}e_i \\ e_i\end{pmatrix},
\end{equation}
where $e_i$ is the $i$-th standard basis vector in ${\mathbb R}^n$. Hence, we can obtain the $i$-th column of $\textbf{A} \textbf{B}$ by computing $v$, extracting its first $n$ entries, and scaling each row -- i.e., multiplying by $(\lambda-1)(\lambda \textbf{I}-\textbf{D}_1)$, which we note requires only an additional $O(n)$ flops. Repeating this for all $n$ eigenvalues of $\textbf{D}_3$ provides the full product $\textbf{A}\textbf{B}$.}

\section{Additional Numerical Experiments}\label{sec:Add_numerical_ex}
This appendix contains a collection of additional numerical examples, focusing on fine-grain tests of the various parameters of block/recursive Jacobi. Implementation details are the same here as in Section~\ref{section: numerical_ex}. \\
\indent Our first test explores how the problem/block size impacts the number of sweeps required by Algorithm~\ref{alg:Jacobi2}. Data for various choices of $n$ and $b$ are reported in the following table, where, as in Figure~\ref{fig:exp_comparison}, we apply block Jacobi to a random symmetric  matrix without applying LUPP or QRCP. 

\renewcommand{\arraystretch}{1.2}
\begin{table}[h!]
\centering
\begin{tabular}{c|cccc}
\toprule
& \multicolumn{4}{c}{\textbf{Partition} ($n/b$)} \\
\textbf{Matrix Size} ($n$) & \hspace{.5pt}4\hspace{.5pt} & \hspace {.5pt} 8 \hspace{.5pt} & 16 & 32  \\
\hline
128 & 4 & 5 & 6 & 6 \\
256 & 4 & 5 & 6 & 6 \\
512 & 4 & 5 & 6 & 6 \\
1024 & 4 & 5 & 6 & 6 \\
2048 & 4 & 5 & 6 & 6 \\
\bottomrule
\end{tabular}
\caption{Number of sweeps required by Algorithm~\ref{alg:Jacobi2} for varying combinations of problem/block size.}
\label{tab:block-jacobi-sweeps}
\end{table}
\renewcommand{\arraystretch}{1}

Here, we observe that the number of sweeps required by block Jacobi does not grow with the problem size $n$. Again, this confirms empirically that our assumption of convergence in $O(1)$ sweeps is realistic when the blocking strategy is fixed; we can therefore expect the trend of Figure~\ref{fig:exp_comparison} -- i.e., that block Jacobi converges in only a handful of sweeps -- to hold even for much larger input matrices. At the same time, we note that increasing the partition size $n/b$ \textit{does} result in an increase, albeit modest, in the number of sweeps required to reach convergence. \\
\indent Next, we consider the impact of the bottom-case diagonalization accuracy on recursive Jacobi, recalling that in the blocked/recursive setting, exactly diagonalizing subproblems is not strictly necessary.  Figure \ref{fig:exp_comparison_bottom} shows the convergence behavior of Algorithm~\ref{alg:Jacobi3} with $n = 512$, $f = 0.4$, 
$n_{\text{threshold}} = 4$, and no pivoting. We include four versions, each of which handles the base case via scalar Jacobi's method (i.e., Algorithm~\ref{alg:Jacobi1}) run to a different final accuracy, again measured by largest off-diagonal entry. As expected, convergence is achieved relatively quickly when the bottom-case precision is strict enough (i.e., at least as strict as the overall convergence criterion), though the method may stagnate otherwise. This demonstrates the sensitivity of the recursive method to base-level accuracy; that is, lower accuracy in the smallest blocks may propagate and ultimately impede global convergence. To present the results comprehensively and assist readers in tracking the convergence history, Figure~\ref{fig:exp_comparison_bottom} additionally plots the corresponding off-diagonal Frobenius norm for each method as a reference. \\
\indent Finally, we consider the impact of recursion depth on Algorithm~\ref{alg:Jacobi3}. Figure~\ref{fig:exp_comparison_maxdepth} plots convergence data for several versions of recursive Jacobi, again for $n = 512$, $n_{\text{threshold}} = 4$ and no pivoting, but this time with $f = 0.8$. In each, a pre-determined maximum recursion depth is enforced by defaulting to a direct diagonalization if reaching the corresponding recursion level. Here, a recursive depth of zero corresponds to the original input matrix and therefore ``max recDepth = 1" reduces Algorithm~\ref{alg:Jacobi3} to Algorithm~\ref{alg:Jacobi2}. Again, we present both the maximum off-diagonal entry and the off-diagonal Frobenius norm for each method. We see here that, for fixed $n$, increasing the recursive depth drives up the number of flops but does not impact convergence.   
\begin{figure}[ht!]
    \centering
    \makebox[\textwidth][c]{%
        \includegraphics[width=\textwidth]{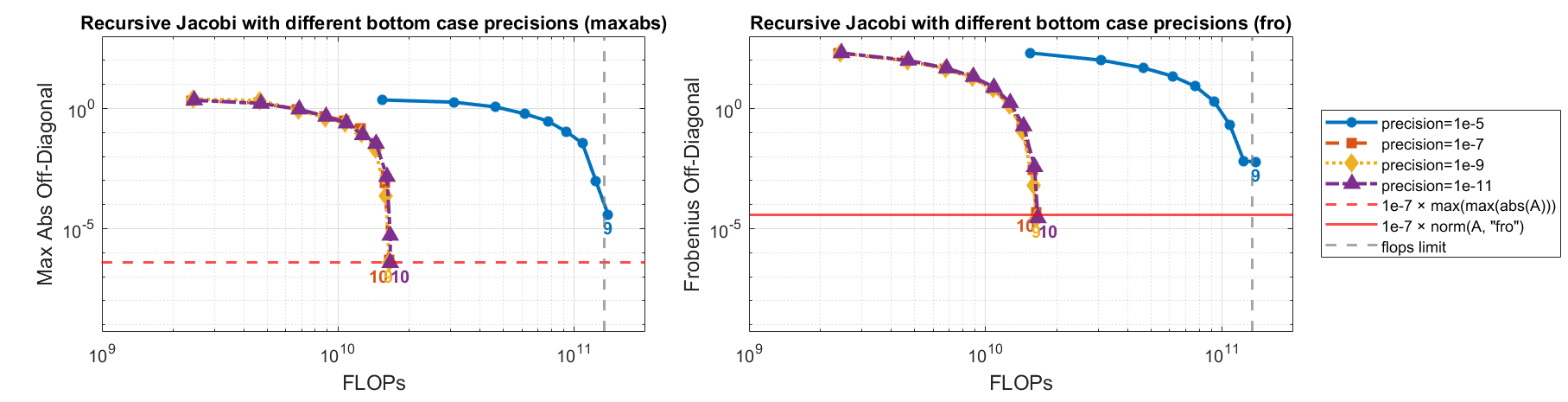}
    }
    \caption{Impact of base-case accuracy on the convergence of recursive Jacobi for a fixed $512 \times 512$ matrix, $f = 0.4$, and $n_{\text{threshold}} = 4$. To vary this accuracy, line 4 of Algorithm~\ref{alg:Jacobi3} is done by scalar Jacobi with different stopping conditions.}
    \label{fig:exp_comparison_bottom}
\end{figure}

\begin{figure}[ht!]    \centering
    \makebox[\textwidth][c]{%
        \includegraphics[width=\textwidth]{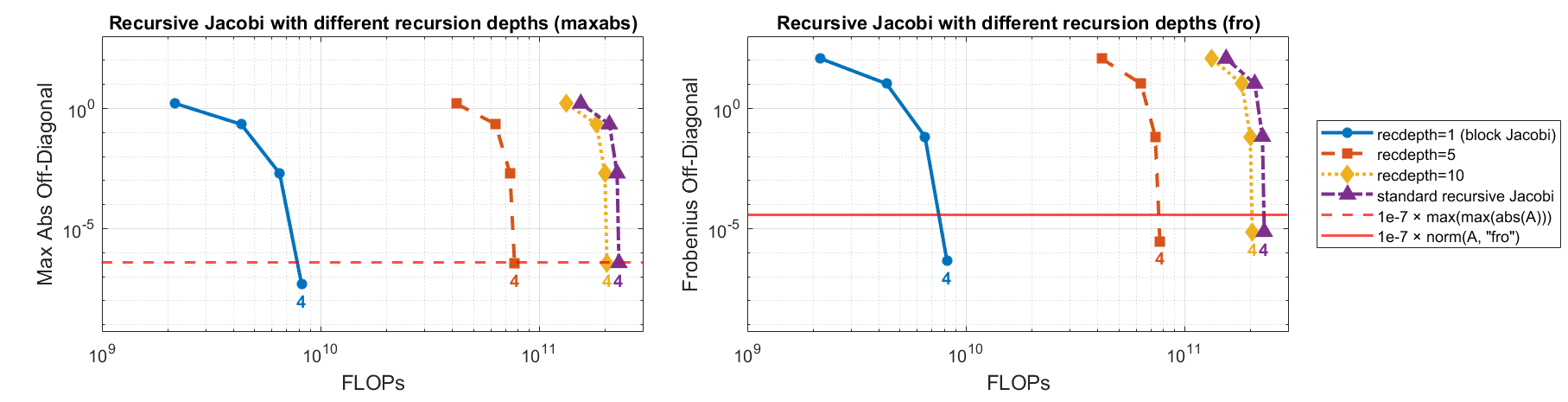}
    }
    \caption{Impact of pre-determined maximum recursion depth on the convergence of recursive Jacobi for a fixed $512 \times 512$ matrix, $f = 0.8$, and $n_{\text{threshold}} = 4$. For this test, recursion depth is controlled by manually defaulting to a direct diagonalization once the maximum is hit.}
    \label{fig:exp_comparison_maxdepth}
\end{figure}

\section{Fast, Recursive LUPP}\label{section: LUPP}
This appendix contains pseudocode for the recursive LUPP routine used to guarantee convergence in blocked/recursive Jacobi (i.e., Algorithm~\ref{alg:LUPP}). 

\begin{algorithm}[h]
\caption{Recursive LU with Partial Pivoting}
\label{alg:LUPP} \begin{algorithmic}[1] \Require $\M{A} \in {\mathbb R}^{m \times n}$ with $m \geq n$
\Ensure $\M{A} = \M{P}^T \M{L} \M{U}$ for $\M{L}$ unit lower trapezoidal, $\M{U}$ upper triangular, and $\M{P}$ a permutation matrix 
\algrule
\Function{$[\M{L},\M{U},\M{P}]=$ Recursive\_LUPP}{$\M{A}$} \State $\M{P} = \M{I}_m$
 \State $[\sim,i] = \max(\M{A}(\; : \;, 1))$
\If{$i \neq 1$}\
    \State Swap rows 1 and $i$ in $\M{A}$ and $\M{P}$ \Comment{Pivot largest entry to the top}
\EndIf
\If{$ n = 1$}
    \State $\M{L} = \M{A}/\M{A}(1)$; $\M{U} = \M{A}(1)$
\Else
    \State $[\M{L}_L, \M{U}_L, \M{P}_L] =$ \textsc{Recursive}\_LUPP($\M{A}(\; : \;  , 1:\lfloor \frac{n}{2} \rfloor$)) \Comment{Apply LUPP to left half of $\M{A}$}
    \State $\M{P} = \M{P}_L \M{P}$
    \setstretch{1.2}
    \State $\M{A}(\; : \; , \lfloor \frac{n}{2} \rfloor + 1:n) = \M{P}_L^T \M{A}(\; : \; ,\lfloor \frac{n}{2} \rfloor + 1:n)$ \Comment{Update right side of $\M{A}$}
    \State $\M{A}(1:\lfloor \frac{n}{2} \rfloor, \lfloor \frac{n}{2} \rfloor + 1 : n) = \M{L}_L(1:\lfloor \frac{n}{2} \rfloor, \; : \;)^{-1}\M{A}(1:\lfloor \frac{n}{2} \rfloor, \lfloor \frac{n}{2} \rfloor + 1 : n)$
    \State $\M{A}(\lfloor \frac{n}{2} \rfloor + 1 : m, \lfloor \frac{n}{2} \rfloor + 1: n) \mathrel{{-}{=}} \M{L}_L(\lfloor \frac{n}{2} \rfloor +1:m, \; : \; ) \M{A}(1:\lfloor \frac{n}{2} \rfloor,\lfloor \frac{n}{2} \rfloor + 1:n)$ 
    \State $[\M{L}_R, \M{U}_R, \M{P}_R] =$ \textsc{Recursive}\_LUPP($\M{A}(\lfloor \frac{n}{2} \rfloor + 1: m, \lfloor \frac{n}{2} \rfloor + 1 
    : n)$) \Comment{Factor lower right block}
    \State $\M{P}(\lfloor \frac{n}{2} \rfloor + 1 : m, \; : \; ) = \M{P}_R \M{P}(\lfloor \frac{n}{2} \rfloor + 1 : m, \; : \; )$ \vskip 3pt
    \State $\M{L} = \begin{bmatrix} \M{L}_L(1:\lfloor \frac{n}{2} \rfloor, \; : \;) & 0 \\
    \M{P}_R \M{L}_L(\lfloor \frac{n}{2} \rfloor +1 : m, \; : \;) & \M{L}_R \end{bmatrix}; \; \; \M{U} = \begin{bmatrix} \M{U}_L & \M{A}(1:\lfloor \frac{n}{2} \rfloor, \lfloor \frac{n}{2} \rfloor + 1: n) \\
    0 & \M{U}_R \end{bmatrix} $ \Comment{Assemble}

\EndIf
\EndFunction \end{algorithmic}
\end{algorithm}

\end{document}